\documentclass[12pt]{article}
\usepackage{amssymb}
\usepackage{amsmath}
\usepackage{amsthm}
\usepackage{float}
\usepackage{graphicx}
\usepackage{multirow}
\usepackage[numbers,sort&compress]{natbib}
\usepackage{color}
\newtheorem{theorem}{Theorem}

\newtheorem{lemma}{Lemma}[section]
\newtheorem{corollary}{Corollary}[section]

\newtheorem{example}{Example}[section]

\renewenvironment{proof}{
\noindent{\bf Proof.}\rm} {\mbox{}\hfill\rule{0.5em}{0.809em}\par}

\begin{document}

\title{On formulas  and some combinatorial properties of Schubert Polynomials\footnote{Supported by the NSFC of China (11571121) and the Science and Technology Program of Guangzhou (201707010137).}}

\author{
Zerui Zhang\footnote {Supported by the Innovation Project of Graduate School of South China Normal University.},\ \
Yuqun
Chen\footnote {Corresponding author.} \\
{\small \ School of Mathematical Sciences, South China Normal
University}\\
{\small Guangzhou 510631, P. R. China}\\
{\small 295841340@qq.com, \small yqchen@scnu.edu.cn}
}

\date{}

\maketitle

\noindent{\bf Abstract:}  By applying a Gr\"{o}bner-Shirshov basis of the symmetric group $S_{n}$, we give two formulas for
 Schubert polynomials,  either of which involves only nonnegative monomials.  We also prove some   combinatorial properties of Schubert polynomials.
As applications,
we give two algorithms to calculate the structure constants for Schubert polynomials, one of which
  depends  on Monk's formula.


\noindent {\bf AMS} Mathematics Subject Classification (2010): 05E05, 05E15, 14M15.




\section{Introduction}\label{Intro}
Schubert polynomials were introduced by  I.N. Bernstein, I.M. Gelfand and S.I. Gelfand  \cite{BGG} and Michel Demazure \cite{Demazure} (in
the context of arbitrary root systems) and were extensively developed by Alain Lascoux, Marcel-Paul
 Sch\"{u}tzenberger \cite{LS1,LS2}.  There are lots of other papers about approaches for the computations of
   Schubert polynomials, for example, Sara C. Billey, William  Jockusch   and
    Richard P. Stanley \cite{BJS1993}, Allen Knutson, Ezra Miller \cite{Grobner geometry},
    Rudolf Winkel \cite{Winkel recursive,Winkel diagram},  Sergey Fomin,  Richard P. Stanley \cite{FS}, for more details, see \cite{Winkel diagram}.

For any $u$ in the symmetric group $S_{n}$, $u$ is associated with a Schubert polynomial, denoted by $ \mathfrak{S}_{u}$. It is well known that
there are identities
$$
\mathfrak{S}_{u}\mathfrak{S}_{v}=\sum_{w} c_{u,v}^{w}\mathfrak{S}_{w},
$$
 where the structure constants $c_{u,v}^{w}$ are  some nonnegative integers \cite{Mac, Fulton}. Though there are algorithms to calculate $c_{u,v}^{w}$, there are still no combinatorial proof. It is an open problem
 to find a combinatorial rule for these coefficients. What people know about these coefficients are limited.   One case where an explicit formula is known  is
 Monk's Formula \cite{Monk}. There are a lot of research on the multiplication of Schubert polynomials, for example, see \cite{SAMI ASSAF,A S,growth diagrams,pieri,Winkel multiplication}.

The aim of this paper is to give another two formulas for Schubert polynomials, to establish some combinatorial properties for Schubert polynomials  and to find  algorithms
 to calculate the structure constants for Schubert polynomials.

Our approach is algebraic,  which is based on a Gr\"{o}bner-Shirshov basis of the symmetric group $S_{n}$
defined by $s_{1},\ldots, s_{n-1}$ subjected to the relations : $s_{i}^{2}=1$,
$s_{i}s_{i+1}s_{i}=s_{i+1}s_{i}s_{i+1}$, $s_{i}s_{j}=s_{j}s_{i}$ ($j>i+1$).
In section 2 we give two formulas for Schubert polynomials. The notations are introduced before the corresponding theorems.  The formulas  take the form:

\textbf{Theorem \ref{formula P}}.  For any $u\in S_{n}$  $(n\geq 2)$, we have
$$\partial_{u}\mathfrak{S}_{w_{0}^{n}}=\sum_{\substack{J_{n-1},\ldots, J_{2},J_{1}  }}X^{u}_{J_{n-1}}X^{u(J_{n-1})}_{J_{n-2}}\cdots X^{u(J_{n-1},J_{n-2},\ldots, J_{2})}_{J_{1}}, $$
where the  summation  is over all the $J_{n-1},\ldots, J_{1}$ such that $ u(J_{n-1},\ldots, J_{l})$ is defined for
any $1\leq l\leq n-1$.

\textbf{Theorem \ref{formula Q}}. For any   $u\in S_{n}$   $(n\geq 2)$, we have
$$\partial_{u}\mathfrak{S}_{w_{0}^{n}}=\sum_{\overrightarrow{T_{n-1}}, \ldots, \overrightarrow{T_{1}}}X_{\overrightarrow{T_{n-1}}}^{u}X_{\overrightarrow{T_{n-2}}}^{u(\overrightarrow{T_{n-1}})}\cdots X_{\overrightarrow{T_{1}}}^{u(\overrightarrow{T_{n-1}},\ldots, \overrightarrow{T_{2}})},$$
where the  summation  is over  all the $\overrightarrow{T_{n-1}}, \ldots, \overrightarrow{T_{1}}$ such that $u(\overrightarrow{T_{n-1}},\ldots,\overrightarrow{T_{l}})$ is defined for
any $1\leq l\leq n-1$.

 One of the corollaries of these two formulas is the well-known  fact that the coefficients of monomials of any Schubert polynomial are nonnegative \cite{Mac,Fulton}. By
Theorem \ref{formula Q},  we develop some combinatorial properties of Schubert polynomials in Section \ref{combinatorial property}.
We analysis how to write down the leading monomial of $\partial_{u}\mathfrak{S}_{w_{0}^{n}}$ with
respect to some order, where $w_{0}^{n}$ is the longest word in $S_{n}$. We also offer an algorithm to find a $u\in S_{n}$ ($n$ large enough) such that
the leading monomial of $\partial_{u}\mathfrak{S}_{w_{0}^{n}}$ is a given commutative word. We show
that
 $\overline{\partial_{t}\mathfrak{S}_u}= \overline{\partial_{t}\overline{\mathfrak{S}_u}}$
 if $\operatorname{deg}_{x_{t}}\overline{\mathfrak{S}_u}>\operatorname{deg}_{x_{t+1}}\overline{\mathfrak{S}_u}$,
  where for any polynomial $f$, $\bar{f}$ means the leading monomial of $f$.
In section \ref{section algorithms}, we show how the properties and formulas we established can
be applied to calculate the structure constants. We also explain how to apply Monk's formula  to the calculation of the structure constants.
As results,
we give two algorithms to calculate the structure constants for the multiplications of  Schubert polynomials.

\section{Two formulas for Schubert polynomials}
The \textbf{symmetric group}, $S_{n}$, consists of all bijections from $\{1,2,\ldots, n\}$ to
itself using composition as the multiplication \cite{Sagan}. It is well known that $S_{n}$ can be
 defined by generators $s_{1},\ldots, s_{n-1}$ with relations: $s_{i}^{2}=1$, $s_{i}s_{i+1}s_{i}=s_{i+1}s_{i}s_{i+1}$, $s_{i}s_{j}=s_{j}s_{i}$ ($j>i+1$),
where  $s_i$ corresponds to the adjacent transposition $(i,i+1)\in S_{n}$ for each $1\leq i\leq  n-1 $.
Let $S=\{s_{i}\mid 1\leq i\leq  n-1\}$ ($s_{1}<s_{2}<\cdots <s_{n-1}$)  and $S^{\ast}$ be a free monoid generated by $S$.
We define the {\textbf{degree lexicographic order}} on $S^{\ast}$ by the following:
for any $u=s_{i_1}s_{i_2}\cdots s_{i_p}, v=s_{j_1}s_{j_2}\cdots s_{j_q}\in S^{\ast}$, where each $s_{i_l},s_{j_t}\in S$,
\begin{equation*}\label{equ0}
u>v \ \Leftrightarrow \ (p,s_{i_1},s_{i_2},\ldots, s_{i_p})>(q,s_{j_1},s_{j_2},\ldots, s_{j_q}) \ \mbox{lexicographically}.
\end{equation*}   We also define the  {\textbf{degree}} of $u$, denoted by $|u|$, to be $p$ if $u=s_{i_1}s_{i_2}\cdots s_{i_p}\in S^{\ast}$.

Using the theory of Gr\"{o}bner-Shirshov bases theory of associative algebras \cite{BokutChenBook},
we know that, under the above definition of $S_{n}$ by generators and relations, $S_{n}$ has a Gr\"{o}bner-Shirshov basis,
with respect to degree lexicographic order on $S^{\ast}$,  as follows:
\begin{enumerate}
\item[(1)] $s_{i}^{2}=1$, $ 1\leq i\leq  n-1 $;
\item[(2)]  $s_{i}s_{j}=s_{j}s_{i}$, $i>j+1, \ 1\leq j<i\leq  n-1 $;
\item[(3)]  $s_{i,j}s_{i}=s_{i-1}s_{i,j}$, $i>j$, where
             $s_{i,j}$ is defined to be $s_{i}s_{i-1}\cdots s_{j}$ if $j\leq i$ and $1$ otherwise ($1$ means the identity element of $S_{n}$).
\end{enumerate}
Then  the follow set
$$
B_{s}^{n}:=\{s_{1,i_{1}}s_{2,i_{2}}\cdots s_{n-1,i_{n-1}}\in S^{\ast} \mid 1\leq i_{j}\leq j+1, 1\leq j \leq  n-1\}
$$
consists of  normal forms of elements of $S_{n}$.
For example, $s_{3,2}s_{5,3}\in B_{s}^{n}$ ($n\geq 6$).

For some historical reason, we also call Gr\"{o}bner basis as Gr\"{o}bner-Shirshov basis for noncommutative cases, for more details, see a survey \cite{survey}.

For any $u\in S^{\ast}$,     we call $u$
a \textbf{{reduced word}}    if for any $v\in S^{\ast}$ with $u=v\in S_{n}$, then $|u|\leq |v|$. For any $u\in S^{\ast}$, let $[u]\in B_{s}^{n}$ be the normal form of $u$ with respect to the above Gr\"{o}bner-Shirshov basis. Since we use degree lexicographic order, we have that $u$ is a reduced word if and only if $|u|=|[u]|$. In other words, we can apply only relations $(2)$ and $(3)$
of the Gr\"{o}bner-Shirshov basis of $S_{n}$ to rewrite $u$ to the normal form $[u]$.
Moreover, the \textbf{length}  of $u$, denoted by  $l(u)$, is defined to be   $|[u]|$. For example, $u=s_{5}s_{4}s_{3}s_{5}s_{4}$ is reduced, for $u$ can be rewritten to $s_{4}s_{3}s_{5}s_{4}s_{3}$ and the latter is a normal form.

From now on, by $u\in S_{n}$, we always assume that $u\in B_{s}^{n}$ unless otherwise specified.

For any $0<n \in \mathbb{N}$, where $\mathbb{N}$ is the set of nonnegative integers, define a group homorphism $\sigma_{n} : S_{n}\longrightarrow S_{n+1}$,
 induced by $s_i\mapsto s_i,\ 1\leq i\leq  n-1$. It is clear that  $\sigma_{n}$ is an embedding, i.e., $S_{n}\subset S_{n+1}$. So we
can define $S_{\infty}=\cup_{n\geq 1} S_{n}$.

Let $\mathbb{Z}[x_1,\ldots,  x_{n}]$ be the free commutative algebra generated
 by $\{x_1,\ldots,  x_{n}\}$ over $\mathbb{Z}$, where $\mathbb{Z}$ is the
 ring of integer numbers. For any polynomial $f\in \mathbb{Z}[x_1,\ldots,  x_{n}]$, for any $i$ between $1$ and $n-1$,
    denote by $s_{i}f$   the result of interchanging $x_{i}$ and $x_{i+1}$ in $f$. Define
the \textbf{divided difference operators} \cite{Fulton} $\partial_{i}$   on the
 polynomial ring $\mathbb{Z}[x_1,\ldots,  x_{n}]$  by the rule:
$$
\partial_{i}(f)=\frac{f-s_{i}f}{x_{i}-x_{i+1}}, \ 1\leq i\leq n-1.
$$
Since  $f-s_{i}f$  is divisible by $x_{i}-x_{i+1}$,  we know that $\partial_{i}(f)$ is still a polynomial in $\mathbb{Z}[x_1,\ldots,  x_{n}]$.
It follows immediately from the definition that for any $f,g\in \mathbb{Z}[x_1,\ldots,  x_{n}]$,
$$
\partial_{i}(fg)=\partial_{i}(f)\cdot g+s_{i}f\cdot \partial_{i}(g).
$$
In particular, if $f=s_{i}f$, then
$\partial_{i}(fg)= f\cdot \partial_{i}(g)$.

Define $\partial_{i}^{0}=id$ (the identity map). Let $p\in \{0,1\}$. Then
$$
\partial_{i}^{p}(fg)=\partial_{i}^{p}(f)\cdot g+ps_{i}f\cdot \partial_{i}(g).
$$
By the definition of $\partial_{t}$ we have that
$$ \partial_{i}(x_{i}^{t}x_{i+1}^{l})=\left\{
\begin{aligned}
&x_{i}^{t-1}x_{i+1}^{l}+x_{i}^{t-2}x_{i+1}^{l+1}+\cdots+x_{i}^{l}x_{i+1}^{t-1} &   & \mbox{ if }  t>l, \\
&0 &   &  \mbox{ if } \ t=l, \\
&-x_{i}^{t}x_{i+1}^{l-1}-x_{i}^{t+1}x_{i+1}^{l-2}-\cdots-x_{i}^{l-1}x_{i+1}^{t} &   &  \mbox{ if } t<l.
\end{aligned}
\right.
$$

Define
 $$
 B_{x}:=\{x_{1}^{k_{1}}\cdots x_{n-1}^{k_{n-1}}\mid k_{i}+i\leq n, 1\leq i\leq n-1\}.
 $$
  Denote  by   $\oplus_{b\in B_{x}}\mathbb{Z}b$ the free $\mathbb{Z}$-module
with $\mathbb{Z}$-basis $B_{x}$.
It follows that for any polynomial $f\in \oplus_{b\in B_{x}}\mathbb{Z}b$,   we have  $\partial_{t} f\in \oplus_{b\in B_{x}}\mathbb{Z}b$ for any $1\leq t\leq n-1$.
The divided difference operators   $\partial_{i}$'s satisfy the
nilCoxeter relations \cite{FS}:
$R=\{\partial_{i}^{2}=0$, $1\leq i\leq n-1$; $\partial_{i}\partial_{i+1}\partial_{i}=\partial_{i+1}\partial_{i}\partial_{i+1}$, $1\leq i\leq n-2$; $\partial_{i}\partial_{j}=\partial_{j}\partial_{i}$, $j>i+1\}$.
It is easy to see that the following relations
\begin{enumerate}
\item[(i)] $\partial_{i}^{2}=0$, $ 1\leq i\leq n-1 $,
\item[(ii)]  $\partial_{i}\partial_{j}=\partial_{j}\partial_{i}$, $i>j+1, \ 1\leq j<i\leq n-1 $,
\item[(iii)]  $\partial_{i,j}\partial_{i}=\partial_{i-1}\partial_{i,j}$, $i>j$, where
             $\partial_{i,j}$ is defined to be  $\partial_{i}\partial_{i-1}\cdots \partial_{j}$ if $j\leq i$ and $id$ otherwise,
\end{enumerate}
form a Gr\"{o}bner-Shirshov basis of the   associative algebra $\mathbb{Z}\langle \partial_{1},\partial_{2},\ldots, \partial_{n-1} |R \rangle$ generated by $\{ \partial_{i} \mid 1\leq i\leq n-1\}$ with relations $R$ over $\mathbb{Z}$.
 This algebra is called  \textbf{nilCoxeter algebra}, and is denoted by $\mathcal{NC}_{n }$.
It follows that a $\mathbb{Z}$-basis of this algebra  is
 $$B_{\partial}:=\{\partial_{1,i_1}\partial_{2,i_2}\cdots \partial_{n-1,i_{n-1}}\mid
 1\leq j\leq n-1, \  1\leq i_j\leq j+1\}.$$
For any  word $u=s_{i_1}s_{i_2}\cdots s_{i_t}\in S^{\ast}$, define
$$
\partial_u:=\partial_{i_1}\partial_{i_2}\cdots \partial_{i_t} \ (\partial_{u}=id \mbox{ if } u =1).
$$
It follows that if $u\in S_{n}\subset S_{\infty}$ is not reduced, then by applying (2) and (3) of the Gr\"{o}bner-Shirshov basis of $S_{n}$ (if necessary),
 $u$ can be rewritten to $v_{1}s_{i}s_{i}v_{2}$ for some $i$.  By applying (ii)
 and (iii) of the Gr\"{o}bner-Shirshov basis of $\mathcal{NC}_{n}$, we
  have   $\partial_u=\partial_{v_{1}s_{i}s_{i}v_{2}}=0$. By similar reasoning, if $u$ and $v$ are
 two reduced words with     $u=v\in S_{n}$, then $\partial_u=\partial_v$.

Let $w_{0}^{n}=s_{1,1}s_{2,1}\cdots s_{n-1,1}\in S_{n}$.    For any $w\in S_{n}$, define
the \textbf{Schubert polynomial} corresponding to $w$ as
$$
\mathfrak{S}_{w}=\mathfrak{S}_{w}(x_{1},\ldots, x_{n-1}):=\partial_{[w^{-1}w_{0}^{n}]}(x_{1}^{n-1}x_{2}^{n-2}\cdots x_{n-2}^{2}x_{n-1}),
$$
where $[w^{-1}w_{0}^{n}]$ is the normal form of $w^{-1}w_{0}^{n}$ with respect to  the above  Gr\"{o}bner-Shirshov basis of $S_{n}$.
In particular, $\mathfrak{S}_{w_{0}^{n}}=x_{1}^{n-1}x_{2}^{n-2}\cdots x_{n-2}^{2}x_{n-1}$ and   $\partial_{u}\mathfrak{S}_{w_{0}^{n}}$ is the Schubert polynomial corresponding to $w_{0}^{n}u^{-1}$.
It is easy to see that, if $w\in S_{n}\subseteq S_{n+1}$, then $[w^{-1}w_{0}^{n+1}]=[w^{-1}w_{0}^{n}]s_{n,1}$, and thus
$\mathfrak{S}_{w}=\partial_{[w^{-1}w_{0}^{n}]}(\mathfrak{S}_{w_{0}^{n}})=\partial_{[w^{-1}w_{0}^{n+1}]}(\mathfrak{S}_{w_{0}^{n+1}})$.
 It is an  not obvious fact   that the coefficients
 of $\partial_{u}\mathfrak{S}_{w_{0}^{n}}$ are nonnegative integers, see \cite{Mac}. We will give two simple formulas for $\partial_{u}\mathfrak{S}_{w_{0}^{n}}$ in the sequel, by either of which follows that the coefficients of $\partial_{u}\mathfrak{S}_{w_{0}^{n}}$
 are nonnegative.

Define
$$ P(i>j):=\left\{
\begin{aligned}
&1 &   &  \mbox{ if }  i>j, \\
&0 &   &  \mbox{otherwise},
\end{aligned}
\right.
$$
$$ P(i\leq j):=\left\{
\begin{aligned}
&1 &   &  \mbox{ if } i\leq j, \\
&0 &   &  \mbox{otherwise},
\end{aligned}
\right.
$$
$$ [i,j]:=\left\{
\begin{aligned}
&\{i,i+1,\ldots, j\} &   &  \mbox{ if }i\leq j, \\
&\emptyset &   &  \mbox{otherwise},
\end{aligned}
\right.
$$
$$ X[i,j]:=\left\{
\begin{aligned}
&x_{i}x_{i+1}\cdots x_{j} &   & \mbox{ if } i\leq j, \\
&1 &   & \mbox{otherwise},
\end{aligned}
\right.
$$
$$ \frac{s_{i,j}}{s_{k}}:=\left\{
\begin{aligned}
&s_{i}\cdots \hat{s_{k}}\cdots s_{j} &   & \mbox{ if } j\leq k\leq i, \\
& s_{i,j} &   &  \mbox{ if } j\leq i< k \mbox{ or } k<j\leq i, \\
& 1 &   &  \mbox{ if } i<j ,
\end{aligned}
\right.
$$
$$ \frac{\partial_{i,j}}{\partial_{k}} :=\left\{
\begin{aligned}
&\partial_{i}\cdots \hat{\partial_{k}}\cdots \partial_{j} &   &  \mbox{ if } j\leq k\leq i, \\
& \partial_{i,j} &   &  \mbox{ if } j\leq i< k \mbox{ or } k<j\leq i, \\
& id &   &  \mbox{ if }i<j ,
\end{aligned}
\right.
$$
$$ \frac{x_{1}^{k_{1}}\cdots x_{n-1}^{k_{n-1}}}{x_{1}^{l_{1}}\cdots x_{n-1}^{l_{n-1}}}:=\left\{
\begin{aligned}
&x_{1}^{k_{1}-l_{1}}\cdots x_{n-1}^{k_{n-1}-l_{n-1}}  &   & \mbox{ if } k_{i}\geq l_{i} \mbox{ for any }
i\in [1,n-1], \\
& \mbox{undefined}  &   &  \mbox{otherwise}.
\end{aligned}
\right.
$$

For any unary  linear operator $\delta_{1}, \delta_{2},\ldots, \delta_{n-1}$  from
$\mathbb{Z}[x_{1}, x_{2},\ldots, x_{n}]$ to $\mathbb{Z}[x_{1}, x_{2},\ldots, x_{n}]$, define
$$
[\delta_{1}x_{1}\delta_{2}x_{2}\cdots \delta_{n-2}x_{n-2}\delta_{n-1}x_{n-1}]_{R}:=\delta_{1}(x_{1}(\delta_{2}(x_{2}(\dots (\delta_{n-2}(x_{n-2}(\delta_{n-1}x_{n-1})))\dots)))) \mbox{ (right normed)}.
$$

\begin{lemma}\label{right normed delta}
For any $t\in [1,n]$, $\delta_{i}\in \{id, \partial_{i}, s_{i}\}$ $(t\leq i\leq n-1)$,
$\delta_{n}\in \{id, \partial_{i}\}$,  we have
$$[\delta_{t}x_{t}\delta_{t+1}x_{t+1}\cdots  \delta_{n}x_{n}]_{R}=p_{t}x_{t}^{q_{t}}x_{t+1}^{q_{t+1}}\cdots x_{n}^{q_{n}},$$
where $p_t, q_{j}\in \{0,1\}$ for any $t\leq j\leq n$.
\end{lemma}
\begin{proof}
Induction on $t$. If $t=n$, it is clear.
Suppose $t\leq n$,
$[\delta_{t}x_{t} \cdots  \delta_{n}x_{n}]_{R}=p_{t}x_{t}^{q_{t}}x_{t+1}^{q_{t+1}}\cdots x_{n}^{q_{n}}$. Then
$[\delta_{t-1}x_{t-1}\delta_{t}x_{t} \cdots  \delta_{n}x_{n}]_{R}
=p_{t}\delta_{t-1}(x_{t-1}x_{t}^{q_{t}}x_{t+1}^{q_{t+1}}\cdots x_{n}^{q_{n}})$$
$.
If $\delta_{t-1}=id$, it is clear. If $\delta_{t-1}= s_{t-1}$, then
$[\delta_{t-1}x_{t-1}\delta_{t}x_{t} \cdots  \delta_{n}x_{n}]_{R}=p_{t}x_{t-1}^{q_{t}}x_{t}x_{t+1}^{q_{t+1}}\cdots x_{n}^{q_{n}}$.
 If $\delta_{t-1}=\partial_{t-1}$, then $[\delta_{t-1}x_{t-1}\delta_{t}x_{t} \cdots  \delta_{n}x_{n}]_{R}=(1-q_{t})p_{t} x_{t+1}^{q_{t+1}}\cdots x_{n}^{q_{n}}=(1-q_{t})p_{t} x_{t-1}^{0}x_{t}^{0}x_{t+1}^{q_{t+1}}\cdots x_{n}^{q_{n}}$.
\end{proof}

\begin{lemma}\label{move Q}
For any $f\in \mathbb{Z}[x_1,\ldots, x_{n}]$, $i,j,k\in \{1,2, \cdots, n-1\}$, we have
\begin{enumerate}
\item[(1)] If $k>i$ or $k<j$ or $j>i$, then
$$\partial_{i,j}(X[1,k]f) = X[1,k]\partial_{i,j}f.$$
\item[(2)] If $i\geq k\geq j$, then
 $$\partial_{i,j}(X[1,k]f)=\sum_{k\leq t\leq i+1}X[1,k-1]x_{i+1}^{P(t>i)}\frac{\partial_{i,j}}{\partial_{t}}f.$$
\end{enumerate}
\end{lemma}
\begin{proof}
(1) is easy, so we just need to prove (2).  Since
$\partial_{i,j}(X[1,k]f)
=\partial_{i,k}(X[1,k]\partial_{k-1,j}f)$, it is sufficient to
show that $\partial_{i,k}(X[1,k]f)=\sum\limits_{k\leq t\leq i+1}X[1,k-1]x_{i+1}^{P(t>i)}\frac{\partial_{i,k}}{\partial_{t}}f.$

Induction on $i-k$.
If $i-k=0$, then  $i=k$ and
\begin{align*}
&\partial_{i,k}(X[1,k]f)&\\
=&X[1,k-1]\partial_{i}(x_{k}f)&\\
=&X[1,k-1]\partial_{i}x_{k}\cdot  f+X[1,k-1]\cdot s_{i}x_{k}\cdot \partial_{i} f&\\
=&X[1,k-1] \cdot \frac{\partial_{i}}{\partial_{k}} f+X[1,k-1]\cdot  x_{i+1}^{P(k+1>i)}\cdot \frac{\partial_{i}}{\partial_{i+1}} f&\\
=&\sum_{k\leq t\leq i+1}X[1,k-1]x_{i+1}^{P(t>i)}\frac{\partial_{i,k}}{\partial_{t}}f.&
\end{align*}
Suppose the result holds for any $i-k<l$. Let $i-k=l\geq 1$. Then
\begin{align*}
&\partial_{i,k}(X[1,k]f)&\\
=&\partial_{i}\partial_{i-1,k}(X[1,k]f)&\\
=&\sum_{k\leq t\leq i}\partial_{i}(X[1,k-1]x_{i}^{P(t>i-1)}\frac{\partial_{i-1,k}}{\partial_{t}}f)&\\
=&\sum_{k\leq t\leq i-1}\partial_{i}(X[1,k-1]\frac{\partial_{i-1,k}}{\partial_{t}}f)+\partial_{i}(X[1,k-1]x_{i}\frac{\partial_{i-1,k}}{\partial_{i}}f)&\\
=&\sum_{k\leq t\leq i-1}X[1,k-1]\frac{\partial_{i,k}}{\partial_{t}}f+X[1,k-1]\cdot \partial_{i}x_{i} \cdot \frac{\partial_{i,k}}{\partial_{i}}f+X[1,k-1]\cdot s_{i}x_{i} \cdot \partial_{i,k}f&\\
=&\sum_{k\leq t\leq i+1}X[1,k-1]x_{i+1}^{P(t>i)}\frac{\partial_{i,k}}{\partial_{t}}f.&
\end{align*}
\end{proof}

Define $s_{0}= s_{0,i}=1$ for any $i\geq 1$. Given $n\in \mathbb{Z}$  ($ n>1$),
for any $u=s_{1,i_{1}} \cdots s_{n-1,i_{n-1}}\in S_{n}$, for any subset
$
J_{n-1}$ of $   \{l\in [1, n-2]\mid i_{l}\leq l\},
$
 define
$$
\delta_{J_{n-1},q}^{u}=\left\{
\begin{aligned}
&s_{q} &   & \mbox{ if } q\in J_{n-1}, \\
&\partial_{q}^{P(i_{q}\leq q)}  &   &\mbox{ if } q\in [1,n-1]\setminus J_{n-1},
\end{aligned}
\right.
$$
$$
X_{J_{n-1}}^{u}=[\delta_{J_{n-1},1}^{u}x_{1}\delta_{J_{n-1},2}^{u}x_{2}\cdots \delta_{J_{n-1},n-1}^{u}x_{n-1}]_{R},
$$
$$
u(J_{n-1})=\left\{
\begin{aligned}
&[s_{r_1,i_{1}} \cdots s_{r_{n-1},i_{n-1}}] &   &  \mbox{ if } s_{r_1,i_{1}} \cdots s_{r_{n-1},i_{n-1}} \mbox{ is reduced and } X_{J_{n-1}}^{u}\neq 0, \\
&\mbox{undefined} &   &  \mbox{otherwise},
\end{aligned}
\right.
$$
where $r_{q}=q$ if $q\in J_{n-1}$, $r_{q}=q-1$ if $q\in [1,n-1]\setminus J_{n-1}$, $[s_{r_1,i_{1}}s_{r_2,i_{2}}\cdots s_{r_{n-1},i_{n-1}}]\in S_{n-1}$ is the normal form of $s_{r_1,i_{1}}s_{r_2,i_{2}}\cdots s_{r_{n-1},i_{n-1}}$.  Remind that for any $w\in S^{\ast}$, $w$ is reduced if and only if $w$ can be rewritten to its normal form by applying $(2)$ and $(3)$ of the Gr\"{o}bner-Shirshov basis of $S_{n}$.

We   proceed to define $u(J_{n-1},\ldots, J_{n-l-1})$ by induction on $l$.
Suppose that $u(J_{n-1},\ldots, J_{n-l})$ has been defined.
 If $u(J_{n-1},\ldots, J_{n-l})$ is undefined, then
 $J_{n-l-1}$, $X_{J_{n-l-1}}^{u(J_{n-1},\ldots, J_{n-l})}$ and
    $u(J_{n-1},\ldots, J_{n-l-1})$ are undefined. Otherwise, say
  $u(J_{n-1},\ldots, J_{n-l}) =s_{1,j_{1}}\cdots s_{n-l-1,j_{n-l-1}}\in S_{n-l}$. Then for any subset
$J_{n-l-1}$ of $  \{k\in [1, n-l-2]\mid j_{k}\leq k\}$, we can define
$X_{J_{n-l-1}}^{u(J_{n-1},\ldots, J_{n-l})}$ and then define
$u(J_{n-1},\ldots, J_{n-l}, J_{n-l-1})$ to be  $(u(J_{n-1},\ldots, J_{n-l}))( J_{n-l-1})$.

Note that $J_{n-1}$ depends on $u$, $J_{n-2}$ depends on $u(J_{n-1})$, and so on.

\begin{theorem}\label{formula P}
 For any $u=s_{1,i_{1}}\cdots s_{n-1,i_{n-1}}\in S_{n}$  $(n\geq 2)$, we have
$$\partial_{u}\mathfrak{S}_{w_{0}^{n}}=\sum_{\substack{J_{n-1},\ldots, J_{2},J_{1} }}X^{u}_{J_{n-1}}X^{u(J_{n-1})}_{J_{n-2}}\cdots X^{u(J_{n-1},J_{n-2},\ldots, J_{2})}_{J_{1}}, $$
where the  summation  is over all the  $J_{n-1},\ldots, J_{1}$ such that $ {u(J_{n-1},\ldots, J_{l})}$ is defined for
any $1\leq l\leq n-1$.
\end{theorem}
\begin{proof}
Induction on $n$.
If $n=2$, then $u=s_{1,i_{1}}$, $i_{1}=$ 1 or 2. For either case, we have  $J_{1}=\emptyset$, $u(J_{1})=1$, $\partial_{u}\mathfrak{S}_{w_{0}^{2}}=\partial_{u}(x_{1})=X[1,i_{1}-1]=X^{u}_{J_{1}}$.

If $n\geq 3$,
we first show that for any $u=s_{t,i_t}s_{t+1,i_{t+1}}\cdots s_{n-1,i_{n-1}}\in S_{n}$,
we have
$$
\partial_{u}\mathfrak{S}_{w_{0}^{n}}
=\sum_{ J_{n-1} \subseteq \{l\in [t,n-2]\mid i_{l}\leq l\}}
X[1,t-1][\delta_{J_{n-1},t}^{u}x_{t}\delta_{J_{n-1},t+1}^{u}x_{t+1}\cdots  \delta_{J_{n-1},n-1}^{u}x_{n-1}]_{R}\cdot
 \partial_{u(J_{n-1})}(\mathfrak{S}_{w_{0}^{n-1}}),
 $$
 where $\delta_{J_{n-1},q}^{u}=s_{q}$ and $ r_{q}=q$ if   $q\in J_{n-1}$;
   $\delta_{J_{n-1},q}^{u}=\partial_{q}^{P(i_{q}\leq q)}$ and $r_{q}=q-1$ if $q\in [t,n-1]\setminus J_{n-1}$;
   $u(J_{n-1})=[s_{r_{t},i_{t}}\cdots s_{r_{n-1},i_{n-1}}]$
 if $s_{r_{t},i_{t}}\cdots s_{r_{n-1},i_{n-1}}$ is reduced and
 $X[1,t-1][\delta_{J_{n-1},t}^{u}x_{t}$ $ \cdots  \delta_{J_{n-1},n-1}^{u}x_{n-1}]_{R}\neq 0$,
 $u(J_{n-1})$ is undefined otherwise and the summation is over all the $J_{n-1}$ such that
 $u(J_{n-1})$ is defined.
Induction on $t$.

Suppose  $t=n-1$. If  $i_{n-1}=n$, then $u=1$, $J_{n-1}\subseteq \emptyset$, $u(J_{n-1})=1$. Therefore
  $$
  \partial_{u}\mathfrak{S}_{w_{0}^{n}}=
  X[1,n-2]  x_{n-1}\cdot \mathfrak{S}_{w_{0}^{n-1}}
  =\sum_{ J_{n-1} \subseteq \{l\in [n-1,n-2]\mid i_{l}\leq l\}}X[1,n-2]\delta_{J_{n-1},n-1}^{u}x_{n-1} \cdot \partial_{u(J_{n-1})} \mathfrak{S}_{w_{0}^{n-1}}.
  $$
If   $i_{n-1}\leq  n-1$, then $u=s_{n-1,i_{n-1}}$, $J_{n-1}\subseteq \emptyset$, $u(J_{n-1})=s_{n-2,i_{n-1}}$. By Lemma \ref{move Q}, we have
\begin{align*}
&\partial_{u}\mathfrak{S}_{w_{0}^{n}}&\\
=&\partial_{n-1} \partial_{n-2,i_{n-1}}(X[1,n-1]\cdot \mathfrak{S}_{w_{0}^{n-1}})&\\
=&\partial_{n-1} (X[1,n-1]\cdot \partial_{n-2,i_{n-1}} \mathfrak{S}_{w_{0}^{n-1}})&\\
=&X[1,n-2]\partial_{n-1} (x_{n-1}\cdot \partial_{n-2,i_{n-1}} \mathfrak{S}_{w_{0}^{n-1}})&\\
=&X[1,n-2]\partial_{n-1} (x_{n-1}) \cdot \partial_{n-2,i_{n-1}} \mathfrak{S}_{w_{0}^{n-1}}
+X[1,n-2] s_{n-1}x_{n-1}\cdot \partial_{n-1}(\partial_{n-2,i_{n-1}} \mathfrak{S}_{w_{0}^{n-1}})&\\
=&X[1,n-2]\partial_{n-1} (x_{n-1}) \cdot \partial_{n-2,i_{n-1}} \mathfrak{S}_{w_{0}^{n-1}}&\\
=&\sum_{ J_{n-1} \subseteq \{l\in [n-1,n-2]\mid i_{l}\leq l\}}
X[1,n-2]\delta_{J_{n-1},n-1}^{u}x_{n-1} \cdot \partial_{u(J_{n-1})} \mathfrak{S}_{w_{0}^{n-1}},&
\end{align*}
for $  \partial_{n-1}(\partial_{n-2,i_{n-1}} \mathfrak{S}_{w_{0}^{n-1}})=0$.

Suppose $t< n-1$. If $i_{t}=t+1$, then  $u=s_{t+1,i_{t+1}}\cdots s_{n-1,i_{n-1}}$,
$\{l\in [t,n-2]\mid i_{l}\leq l\}=\{l\in [t+1,n-2]\mid i_{l}\leq l\}$. So
\begin{align*}
&\partial_{u}\mathfrak{S}_{w_{0}^{n}}&\\
=&\partial_{s_{t+1,i_{t+1}}\cdots s_{n-1,i_{n-1}}}\mathfrak{S}_{w_{0}^{n}}&\\
=&\sum_{ J_{n-1} \subseteq \{l\in [t+1,n-2]\mid i_{l}\leq l\}}
X[1,t][\delta_{J_{n-1},t+1}^{u}x_{t+1}\delta_{J_{n-1},t+2}^{u}x_{t+2}\cdots  \delta_{J_{n-1},n-1}^{u}x_{n-1}]_{R}\cdot
 \partial_{u(J_{n-1})}\mathfrak{S}_{w_{0}^{n-1}}&\\
 =&\sum_{ J_{n-1} \subseteq \{l\in [t,n-2]\mid i_{l}\leq l\}}
X[1,t-1][\delta_{J_{n-1},t}^{u}x_{t}\delta_{J_{n-1},t+1}^{u}x_{t+1}\cdots  \delta_{J_{n-1},n-1}^{u}x_{n-1}]_{R}\cdot
 \partial_{u(J_{n-1})}\mathfrak{S}_{w_{0}^{n-1}}.&
\end{align*}

If $i_{t}\leq t$, let $u_{1}=s_{t+1,i_{t+1}}\cdots s_{n-1,i_{n-1}}$. Then
\begin{align*}
&\partial_{u}\mathfrak{S}_{w_{0}^{n}}&\\
=&\partial_{t,i_{t}}\partial_{ u_{1}}\mathfrak{S}_{w_{0}^{n}}&\\
=&\partial_{t,i_{t}}(\sum_{ J_{n-1} \subseteq \{l\in [t+1,n-2]\mid i_{l}\leq l\}}
X[1,t][\delta_{J_{n-1},t+1}^{u_1}x_{t+1}\cdots  \delta_{J_{n-1},n-1}^{u_1}x_{n-1}]_{R}\cdot
 \partial_{u_{1}(J_{n-1})}\mathfrak{S}_{w_{0}^{n-1}})&\\
 =&\partial_{t}(\sum_{ J_{n-1} \subseteq \{l\in [t+1,n-2]\mid i_{l}\leq l\}}
X[1,t][\delta_{J_{n-1},t+1}^{u_1}x_{t+1}\cdots  \delta_{J_{n-1},n-1}^{u_1}x_{n-1}]_{R}\cdot
 \partial_{t-1,i_{t}}\partial_{u_{1}(J_{n-1})}\mathfrak{S}_{w_{0}^{n-1}})&\\
  =&X[1,t-1]\partial_{t}(\sum_{ J_{n-1} \subseteq \{l\in [t+1,n-2]\mid i_{l}\leq l\}}
x_{t}[\delta_{J_{n-1},t+1}^{u_1}x_{t+1}\cdots  \delta_{J_{n-1},n-1}^{u_1}x_{n-1}]_{R}\cdot
 \partial_{t-1,i_{t}}\partial_{u_{1}(J_{n-1})}\mathfrak{S}_{w_{0}^{n-1}})&\\
  =&\sum_{ J_{n-1} \subseteq \{l\in [t+1,n-2]\mid i_{l}\leq l\}}
X[1,t-1]\partial_{t}(x_{t}[\delta_{J_{n-1},t+1}^{u_1}x_{t+1}\cdots  \delta_{J_{n-1},n-1}^{u_1}x_{n-1}]_{R})\cdot
 \partial_{t-1,i_{t}}\partial_{u_{1}(J_{n-1})}\mathfrak{S}_{w_{0}^{n-1}}&\\
+ &\sum_{ J_{n-1} \subseteq \{l\in [t+1,n-2]\mid i_{l}\leq l\}}
X[1,t-1]s_{t}(x_{t}[\delta_{J_{n-1},t+1}^{u_1}x_{t+1}\cdots  \delta_{J_{n-1},n-1}^{u_1}x_{n-1}]_{R})\cdot
 \partial_{t}\partial_{t-1,i_{t}}\partial_{u_{1}(J_{n-1})}\mathfrak{S}_{w_{0}^{n-1}}&\\
 =&\sum_{ J_{n-1} \subseteq \{l\in [t,n-2]\mid i_{l}\leq l\}}
X[1,t-1][\delta_{J_{n-1},t}^{u}x_{t}\delta_{J_{n-1},t+1}^{u}x_{t+1}\cdots  \delta_{J_{n-1},n-1}^{u}x_{n-1}]_{R}\cdot
 \partial_{u(J_{n-1})}\mathfrak{S}_{w_{0}^{n-1}},&
\end{align*}
 where the summation is over all the $J_{n-1}$ such that
 $u(J_{n-1})$ is defined.

Let  $t=1$. Then we have
$$
\partial_{u}\mathfrak{S}_{w_{0}^{n}}
=\sum_{J_{n-1} \subseteq \{l\in [1,n-2]\mid i_{l}\leq l\}}X^{u}_{J_{n-1}}\partial_{u(J_{n-1})}(\mathfrak{S}_{w_{0}^{n-1}}).
$$
By induction hypothesis, we have
$$
\partial_{u}\mathfrak{S}_{w_{0}^{n}}=\sum_{\substack{J_{n-1},\ldots, J_{2},J_{1}  }}X^{u}_{J_{n-1}}X^{u(J_{n-1})}_{J_{n-2}}\cdots X^{u(J_{n-1},J_{n-2},\ldots, J_{2})}_{J_{1}},
$$
where the  summation  is over all the $J_{n-1},\ldots, J_{1}$ such that $ {u(J_{n-1},\ldots, J_{l})}$ is defined for
any $1\leq l\leq n-1$.
\end{proof}

\begin{example}\label{example for formula P}
$\partial_{u}\mathfrak{S}_{w_{0}^{6}}=X[1,5]\cdot \sum\limits_{1\leq i<j\leq 4} x_{i}x_{j}$, where $u=s_{1,1}s_{2,1} s_{3,1} s_{4,3}$.
\end{example}
\begin{proof}
Since we are given $\partial_{u}\mathfrak{S}_{w_{0}^{6}}$, we should begin with $u\in S_{6}$. For $u=s_{1,1}s_{2,1} s_{3,1} s_{4,3} \in S_{6}$, $J_{5}\subseteq \{1,2,3,4\}$. It is easy to see that if $J_{5}\neq \{1,2,3,4\}$, then $X_{J_{5}}^{u}=0$, so $u(J_{5})$ is undefined.  Let
$J_{5}=\{1,2,3,4\}$. Then   $X_{J_{5}}^{u}=X[1,5]$,
$u(J_{5})=u=s_{1,1}s_{2,1} s_{3,1} s_{4,3}\in S_{5}$, $J_{4}\subseteq \{1,2,3\}$. It is straightforward to see that only when $J_{4}=\emptyset$ or $J_{4}=\{3\}$ or
$J_{4}=\{2,3\}$, we have $u(J_{5},J_{4})$ is defined. For example, if $J_{4}=\{2\}$, then $X_{J_{4}}^{u(J_{5})}=x_{3}$, $s_{0,1}s_{2,1} s_{2,1} s_{3,3}$ is not reduced. If $J_{4}=\emptyset$, then $X_{J_{4}}^{u(J_{5})}=1$, $u(J_{5},J_{4})=s_{1,1}s_{2,1}s_{3,3}\in S_{4}$, $J_{3}\subseteq \{1,2\}$. In this way, we can list all the possible $J_{5},\ldots, J_{1}$ such that $u(J_{5},\ldots, J_{1})$ is defined (Table \ref{table 1}).
The result follows immediately from  Theorem \ref{formula P}.
\begin{table}[H]

\begin{tabular}
{|c|c|c|c|c|}

\hline

\multirow{24}{2.15cm}{\small{$J_{5}=\{1,2,3,4\}$,    $X_{J_{5}}^{u}=X[1,5]$,
$u(J_{5}) =s_{1,1}s_{2,1} s_{3,1} s_{4,3}$ $\in S_{5}$ }   } & \multirow{12}{2.5cm}{\small{$J_{4}=\emptyset$, $X_{J_{4}}^{u(J_{5})}=1$, $u(J_{5},J_{4})=s_{1,1}s_{2,1}s_{3,3}$ $\in S_{4}$}}   &  \multirow{4}{3cm}{$J_{3}=\{1,2\}$, $X_{J_{3}}^{u(J_{5},J_{4})}$ $=x_{2}x_{3}$, $u(J_{5},J_{4},J_{3})=s_{1,1}s_{2,1}$ $\in S_{3}$
}    & \multirow{4}{3cm}{
$J_{2}=\emptyset$, $X_{J_{2}}^{u(J_{5},J_{4},J_{3})}=1$, $u(J_{5},\ldots, J_{2})=s_{1,1} $ $\in S_{2}$}   &  \multirow{4}{3cm}{$J_{1}=\emptyset, X_{J_{1}}^{u(J_{5},\ldots, J_{2})}=1, u(J_{5},\ldots, J_{1})=1 $ $\in S_{1}$}     \\

 & & & &  \\

 & & & &  \\

 & & & &  \\

  \cline{3-5}
 &    &  \multirow{4}{3cm}{$J_{3}=\{ 2\}$, $X_{J_{3}}^{u(J_{5},J_{4})}=x_3$, $u(J_{5},J_{4},J_{3})= s_{2,1}$ $\in S_{3}$
}    & \multirow{4}{3cm}{$J_{2}=\emptyset$, $X_{J_{2}}^{u(J_{5},J_{4},J_{3})}=x_{1}$, $u(J_{5},\ldots, J_{2})=s_{1,1}$ $\in S_{2}$}   &  \multirow{4}{3cm}{$J_{1}=\emptyset$, $X_{J_{1}}^{u(J_{5},\cdots ,J_{2})}=1$, $u(J_{5},\ldots, J_{1})=1$ $\in S_{1}$}     \\

 & & & &  \\

 & & & &  \\

 & & & &  \\

  \cline{3-5}
 &    &  \multirow{4}{3cm}{
$J_{3}=\emptyset$, $X_{J_{3}}^{u(J_{5},J_{4})}=1$, $u(J_{5},J_{4},J_{3})= s_{1,1}  $ $\in S_{3}$
}    & \multirow{4}{3.2cm}{\small{
$J_{2}=\{1\}$, $X_{J_{2}}^{u(J_{5},J_{4},J_{3})}$$=x_{1}x_{2}$, $u(J_{5},\ldots, J_{2})=s_{1,1}$ $\in S_{2}$
}}   &  \multirow{4}{3cm}{$J_{1}=\emptyset$, $X_{J_{1}}^{u(J_{5},\ldots, J_{2})}=1$, $u(J_{5},\ldots, J_{1})=1 \in S_{1}$}     \\

 & & & &  \\

 & & & &  \\

 & & & &  \\

  \cline{2-5}
 & \multirow{8}{2.5cm}{\small{$J_{4}=\{3\}$, $X_{J_{4}}^{u(J_{5})}=x_{4}$, $u(J_{5},J_{4})=s_{1,1}s_{2,2}s_{3,1}$ $\in S_{4}$}}   &  \multirow{4}{3cm}{$J_{3}=\{1\}$, $X_{J_{3}}^{u(J_{5},J_{4})}=x_2$, $u(J_{5},J_{4},J_{3})=s_{1,1} s_{2,1}$   $\in S_{3}$
}    & \multirow{4}{3cm}{$J_{2}=\emptyset$, $X_{J_{2}}^{u(J_{5},J_{4},J_{3})}=1$, $u(J_{5},\ldots, J_{2})=s_{1,1}$ $\in S_{2}$
}   &  \multirow{4}{3cm}{
$J_{1}=\emptyset$, $X_{J_{1}}^{u(J_{5},\ldots, J_{2})}=1$, $u(J_{5},\ldots, J_{1})=1$ $\in S_{1}$}   \\

 & & & &  \\

 & & & &  \\

 & & & &  \\

 \cline{3-5}
  &    &  \multirow{4}{3cm}{
$J_{3}=\emptyset$, $X_{J_{3}}^{u(J_{5},J_{4})}=1$, $u(J_{5},J_{4},J_{3})=  s_{2,1}$ $\in S_{3}$
}    & \multirow{4}{3cm}{
$J_{2}=\emptyset$, $X_{J_{2}}^{u(J_{5},J_{4},J_{3})}=x_{1}$, $u(J_{5},J_{4},J_{3},J_{2})=s_{1,1}$ $\in S_{2}$
}   &  \multirow{4}{3cm}{$J_{1}=\emptyset$, $X_{J_{1}}^{u(J_{5},\ldots, J_{2})}=1$, $u(J_{5},\cdots ,J_{1})=1$ $\in S_{1}$}   \\

 & & & &  \\

 & & & &  \\

 & & & &  \\

 \cline{2-5}

   & \multirow{4}{2.5cm}{\small{$J_{4}=\{2, 3\}$, $X_{J_{4}}^{u(J_{5})}=x_{3}x_{4}$, $ u(J_{5},J_{4})=s_{1,1}s_{2,1}s_{3,1}$ $\in S_{4}$}}   &  \multirow{4}{3cm}{$J_{3}=\emptyset$, $X_{J_{3}}^{u(J_{5},J_{4})}=1$, $u(J_{5},J_{4},J_{3})= s_{1,1} s_{2,1}$ $\in S_{3}$
}    & \multirow{4}{3cm}{
$J_{2}=\emptyset$, $X_{J_{2}}^{u(J_{5},J_{4},J_{3})}=1$, $u(J_{5},\ldots, J_{2})=s_{1,1}$ $\in S_{2}$}   &  \multirow{4}{3cm}{$J_{1}=\emptyset$, $X_{J_{1}}^{u(J_{5},\ldots, J_{2})}=1$, $u(J_{5},\ldots, J_{1})=1$ $\in S_{1}$}  \\

 & & & &  \\

 & & & &  \\

 & & & &  \\

 & & & &  \\

\hline

\end{tabular}
\caption{Example \ref{example for formula P}}\label{table 1}
\end{table}
\end{proof}

\ \

Now we begin to construct another formula for Schubert polynomials.

Given $n\geq 2$, for any $u=s_{1,i_{1}}s_{2,i_{2}}\cdots s_{n-1,i_{n-1}}\in S_{n}$,  for any $1\leq r\leq n-1$, define
$$q_{r,0}^{u}=r,\ M_{r,1}^{u}=\{j\in \mathbb{N}\mid 1\leq j<q_{r,0}^{u}, i_{r}-1\in [i_{j}, j]\}.$$
We proceed to define $q_{r,l-1}^{u},\ M_{r,l}^{u}$ by induction on $l$. Suppose that $q_{r,l-1}^{u},\ M_{r,l}^{u}$ have been defined.
If $M_{r,l}^{u}\neq\emptyset$, then define
 $$q_{r,l}^{u}=\operatorname{max}M_{r,l}^{u},\ M_{r,l+1}^{u}=\{j\in \mathbb{N}\mid 1\leq j<q_{r,l}^{u}, i_{r}-(l+1)\in [i_{j}, j]\}.$$
If $M_{r,l}^{u}=\emptyset$, then define
$$m_{r}^{u}=l-1.$$
 Finally, define
 $$Q_{r}^{u}=\{q_{r,j}^{u}\mid 1\leq j\leq m_{r}^{u}\}.$$
In other word, $q_{r,j}^{u}$ $(j\in [1,m_{r}^{u}])$ is the largest integer $q$ such that
$i_{r}-j\in [i_{q},q]$ and $q<q_{r,j-1}^{u}$.
 By definition, we have $Q_{r}^{u}\subseteq [1,r]$, $q_{r,1}^{u}>q_{r,2}^{u}>\cdots$. In particular,
 if $m_{r}^{u}=0$, then $Q_{r}^{u}=\emptyset$. For example,
if $u=s_{3,2}s_{5,1}s_{6,4}s_{7,1}s_{8,8}s_{9,5}$, then
 $q_{9,1}^{u}=7$,  $q_{9,2}^{u}=5$, $q_{9,3}^{u}=3$, $m_{9}^{u}=3$.

\begin{lemma}\label{property Q}
For any $u=s_{1,i_{1}}s_{2,i_{2}}\cdots s_{n-1,i_{n-1}}, \ v=s_{1,j_{1}}s_{2,j_{2}}\cdots s_{n-1,j_{n-1}}\in S_{n}$, we have
\begin{enumerate}
\item[\emph{(1)}]  If $i_{k}=j_{k}$ for any $k\leq r$, then $Q_{t}^{u}=Q_{t}^{v}$, $q_{t,j}^{u}=q_{t,j}^{v}$  for any $1\leq t\leq r$, $1\leq j\leq m_{t}^{u}.$
\item[\emph{(2)}] If $i_{k}=j_{k}$ for any $k\geq r$, then $Q_{t}^{u}\cap [r,n-1]=Q_{t}^{v}\cap [r,n-1]$  for any $r\leq t\leq n-1$. Moreover,
    $q_{t,j}^{u}=q_{t,j}^{v}$ if  $q_{t,j}^{u}\geq r$.
\item[\emph{(3)}]  If  $i_{k}=j_{k}$ for any $k\neq t-1, t$ and $\{t-1,t\}\cap Q_{r}^{u}=\emptyset$, $\{t-1,t\}\cap Q_{r}^{v}=\emptyset $ for some $1\leq r\leq n-1$, then
   $Q_{r}^{u}=Q_{r}^{v}$, $q_{r,j}^{u}=q_{r,j}^{v}$  for any  $1\leq j\leq m_{r}^{u}.$
\item[\emph{(4)}]  For any $j\geq 1$, if  $i_{j}<j+1$, then   $q_{j,m_{j}^{u}}^{u}>i_{j}-m_{j}^{u}-1$; If
         $i_{j}=j+1$, then $q_{j,m_{j}^{u}}^{u}=i_{j}-m_{j}^{u}-1=j$.
\end{enumerate}
\end{lemma}
\begin{proof}
Since (1)-(3) follow immediately from the definition of $q_{t,j}^{u}$,
we just need to prove  (4). If $i_{j}=j+1$, the claim
is easy. So we may assume that
$i_{j}\leq j$.
 If $m_{j}^{u}=0$, then
$q_{j,m_{j}^{u}}^{u}=q_{j,0}^{u}=j>j-1\geq i_{j}-1=i_{j}-m_{j}^{u}-1$.
If $m_{j}^{u}\neq 0$, then $i_{j}-m_{j}^{u}\in [i_{q_{j,m_{j}^{u}}^{u}},q_{j,m_{j}^{u}}^{u}]$.
So $i_{j}-m_{j}^{u}-1<i_{j}-m_{j}^{u}\leq q_{j,m_{j}^{u}}^{u}$.
\end{proof}
Given $n\geq 2$,
  $u=s_{1,i_{1}}s_{2,i_{2}}\cdots s_{n-1,i_{n-1}}\in S_{n}$,
 for any
 $$\overrightarrow{T_{n-1}}\in\{(t_{m_{n-1}^{u}},\ldots, t_{2},  t_{1}, 0 ) \mid   i_{n-1}-j\leq t_{j}\leq q_{n-1,j}^{u}+1, 1\leq j\leq m_{n-1}^{u}, t_{j}\in \mathbb{N}\},$$
define
$$
 u(\overrightarrow{T_{n-1}})=\left\{
\begin{aligned}
&\left[\frac{s_{1,i_{1}}s_{2,i_{2}}\cdots s_{n-2,i_{n-2}}}{(s_{t_{m_{n-1}^{u}}},\ldots, s_{t_{2}},s_{t_{1}} )}\right] &   &  \mbox{ if }  \frac{s_{1,i_{1}}s_{2,i_{2}}\cdots s_{n-2,i_{n-2}}}{(s_{t_{m_{n-1}^{u}}},\ldots, s_{t_{2}},s_{t_{1}} )} \mbox{ is reduced}, \\
&\mbox{undefined} &   &  \mbox{otherwise},
\end{aligned}
\right.
$$
where $\left[\frac{s_{1,i_{1}}s_{2,i_{2}}\cdots s_{n-2,i_{n-2}}}{(s_{t_{m_{n-1}^{u}}},\ldots, s_{t_{2}},s_{t_{1}} )}\right]$ is the normal form of the word getting by substituting every $s_{q_{r,j}^{u},i_{q_{r,j}^{u}}}$ by $\frac{s_{q_{r,j}^{u},i_{q_{r,j}^{u}}}}{s_{t_{j}}}$, $1\leq j\leq m_{n-1}^{u}$.
For each $\overrightarrow{T_{n-1}}$ such that $u(\overrightarrow{T_{n-1}})$ is defined, define
\begin{align*}
X_{\overrightarrow{T_{n-1}}}^{u}
:=&X[1,i_{n-1}-m_{n-1}^{u}-1]\prod_{1\leq j\leq m_{n-1}^{u}}x^{P(t_{j}>q_{n-1,j}^{u})}_{1+q_{_{n-1,j}}^{u}}&\\
=&X[1,i_{n-1}-m_{n-1}^{u}-1]x^{P(t_{m_{n-1}^{u}}>q_{n-1,m_{n-1}^{u}})}_{_{_{1+q_{_{n-1,m_{n-1}^{u}}}^{u}}}}\cdots   x^{P(t_{2}>q_{n-1,2}^{u})}_{_{1+q_{_{n-1,2}}^{u}}}x^{P(t_{1}>q_{n-1,1}^{u})}_{_{1+q_{_{n-1,1}}^{u}}},&
\end{align*}
where $\prod\limits_{1\leq j\leq m_{n-1}^{u}}x^{P(t_{j}>q_{n-1,j}^{u})}_{1+q_{_{n-1,j}}^{u}}=1$ if $m_{n-1}^{u}=0$.
We proceed to define
 $u(\overrightarrow{T_{n-1}}, \cdots, \overrightarrow{T_{n-l-1}})$ by induction on $l$.
Suppose that   $u(\overrightarrow{T_{n-1}}, \cdots, \overrightarrow{T_{n-l}})$  has been  defined. If
 $u(\overrightarrow{T_{n-1}}, \cdots, \overrightarrow{T_{n-l}})$ is undefined, then  $\overrightarrow{T_{n-l-1}}$
 and
 $u(\overrightarrow{T_{n-1}}, \cdots, \overrightarrow{T_{n-l-1}})$ are undefined. Otherwise,
$
 u(\overrightarrow{T_{n-1}}, \cdots, \overrightarrow{T_{n-l}})\in S_{n-l}.
 $ Say $u(\overrightarrow{T_{n-1}}, \cdots, \overrightarrow{T_{n-l}})=v=s_{1,j_1}\cdots s_{n-l-1,j_{n-l-1}}$.
For any vector
  $$\overrightarrow{T_{n-l-1}} \in \{(t_{m_{n-l-1}^{v}},\ldots, t_{2},  t_{1},0 )
  \mid   j_{n-l-1}-j\leq t_{j}\leq q_{n-l-1,j}^{v}+1, 1\leq j\leq m_{n-l-1}^{v},t_{j}\in \mathbb{N}\},$$
  we   define
  $u(\overrightarrow{T_{n-1}}, \cdots, \overrightarrow{T_{n-l-1}})$ to be   $ u(\overrightarrow{T_{n-1}}, \cdots, \overrightarrow{T_{n-l}})(\overrightarrow{T_{n-l-1}})=v(\overrightarrow{T_{n-l-1}}),$
  $X_{\overrightarrow{T_{n-l-1}}}^{u(\overrightarrow{T_{n-1}}, \cdots, \overrightarrow{T_{n-l}})}=X_{\overrightarrow{T_{n-l-1}}}^{v}$.
Note that the set of $\overrightarrow{T_{n-l-1}}$'s  depends on $u(\overrightarrow{T_{n-1}}, \cdots, \overrightarrow{T_{n-l}})$.
However, for simplicity, we just use the notation $\overrightarrow{T_{n-l-1}}$.

In particular, if $Q_{n-1}^{u}=\emptyset  $, then
$
X_{\overrightarrow{T_{n-1}}}^{u}=
X[1,i_{n-1}-m_{n-1}^{u}-1].
$
If $t_{j}= q_{n-1,j}^{u}+1$  for any $1\leq j\leq m_{n-1}^{u}$, then
$
u(\overrightarrow{T_{n-1}})=s_{1,i_{1}}s_{2,i_{2}}\cdots s_{n-2,i_{n-2}}\in S_{n-1}
$ and
$
X_{\overrightarrow{T_{n-1}}}^{u}=X[1,i_{n-1}-m_{n-1}^{u}-1]x_{_{1+q_{_{n-1,m_{n-1}^{u}}}^{u}}}\cdots   x_{_{1+q_{_{n-1,2}}^{u}}}x_{_{1+q_{_{n-1,1}}^{u}}}
$.

For example, we first fix $n=11$. Let $u=s_{3,2}s_{5,1}s_{6,4}s_{7,1}s_{8,8}s_{9,5}\in S_{11}$. Then
$Q_{10}^{u}=\emptyset$, $\overrightarrow{T_{10}}=(0)$, $u(\overrightarrow{T_{10}})=u\in S_{10}$,
$X^{u}_{\overrightarrow{T_{10}}}=X[1,10]$,  $Q_{9}^{u(\overrightarrow{T_{10}})}=\{7,5,3\}$,
 $\overrightarrow{T_{9}}\in \{(t_{3},t_{2},t_{1},0)\in \mathbb{N}^{4}\mid 2\leq t_{3}\leq 4, 3\leq t_{2}\leq 6, 4 \leq t_{1}\leq 8  \}$.
  If $\overrightarrow{T_{9}}=(2,3,4,0)$, then
$u(\overrightarrow{T_{10}})(\overrightarrow{T_{9}})
=s_{3,3}\cdot s_{5,4}s_{2,1}\cdot s_{6,4}\cdot s_{7,5}s_{3,1}\cdot s_{8,8}
=s_{3,1}s_{5,4} s_{6,1}s_{7,5}s_{8,8}
 $,
  $Q_{8}^{u(\overrightarrow{T_{10}})(\overrightarrow{T_{9}})}=\{7,6,5\}$,
$\overrightarrow{T_{8}}\in \{(t_{3},t_{2},t_{1},0)\in \mathbb{N}^{4}\mid 5\leq t_{3}\leq 6, 6\leq t_{2}\leq 7, 7\leq t_{1}\leq 8\}$.

\begin{theorem}\label{formula Q}
For any $u\in S_{n}$  $(n\geq 2)$, we have
$$\partial_{u}\mathfrak{S}_{w_{0}^{n}}=\sum_{\overrightarrow{T_{n-1}}, \cdots, \overrightarrow{T_{1}}}X_{\overrightarrow{T_{n-1}}}^{u}X_{\overrightarrow{T_{n-2}}}^{u(\overrightarrow{T_{n-1}})}\cdots X_{\overrightarrow{T_{1}}}^{u(\overrightarrow{T_{n-1}},\ldots, \overrightarrow{T_{2}})},$$
where the  summation  is over all the  $\overrightarrow{T_{n-1}}, \cdots, \overrightarrow{T_{1}}$ such that $u(\overrightarrow{T_{n-1}},\ldots, \overrightarrow{T_{l}})$ is defined for
any $1\leq l\leq n-1$.
\end{theorem}

\begin{proof}
We first show that for any $u=s_{1,i_1}\cdots s_{n,i_{n}}\in S_{n+1}$, we have
$
\partial_{u}\mathfrak{S}_{w_{0}^{n+1}} =\sum\limits_{\overrightarrow{T_{n}}}X_{\overrightarrow{T_{n}}}^{u}\partial_{u(\overrightarrow{T_{n}})}\mathfrak{S}_{w_{0}^{n}}
$. Induction on $n+1$.

If $n+1=2$, then $u=s_{1,i_{1}}$, $i_{1}=1$ or $2$. For either case, we have $Q_{1}^{u}=\emptyset$, so
$u(\overrightarrow{T_{1}})=1$  and $X_{\overrightarrow{T_{1}}}^{u}=X[1,i_{1}-1]=\partial_{u}(x_{1})=\partial_{u}\mathfrak{S}_{w_{0}^{2}}$.

If $n +1 \geq 3$, then
induction on $m_{n}^{u}$.

If $m_{n}^{u}=0$, then $Q_{n}^{u}=\emptyset$, $\overrightarrow{T_{n}}\in \{(0)\}$,  $u(\overrightarrow{T_{n}})=s_{1,i_{1}}s_{2,i_{2}}\cdots s_{n-1,i_{n-1}}$,
$X_{\overrightarrow{T_{n}}}^{u}=X[1,i_{n}-1]$.
  By applying Lemma \ref{move Q} repeatedly, we have
\begin{align*}
 &\partial_{u}\mathfrak{S}_{w_{0}^{n+1}}&\\
 =&\partial_{s_{1,i_{1}} \cdots s_{n-1,i_{n-1}}}(X[1,i_{n}-1]\mathfrak{S}_{w_{0}^{n}})&\\
 =&\partial_{s_{1,i_{1}} \cdots s_{n-2,i_{n-2}}}(X[1,i_{n}-1] \cdot \partial_{s_{n-1,i_{n-1}}}\mathfrak{S}_{w_{0}^{n}})&\\
  =&\cdots&\\
 =&X[1,i_{n}-1]\partial_{s_{1,i_{1}}s_{2,i_{2}}\cdots s_{n-1,i_{n-1}}}\mathfrak{S}_{w_{0}^{n}}&\\
 =&X_{\overrightarrow{T_{n}}}^{u}\partial_{u(\overrightarrow{T_{n}})}\mathfrak{S}_{w_{0}^{n}}.&
\end{align*}

If $m_{n}^{u}=r\geq 1$, then
$u=s_{1,i_{1}}\cdots s_{q_{_{n,r}}^{u},i_{q_{_{n,r}}^{u}}}s_{q_{_{n,r}}^{u}+1,i_{q_{_{n,r}}^{u}+1}}\cdots s_{n,i_{n}}$.
Let $v=s_{q_{_{n,r}}^{u}+1,i_{q_{_{n,r}}^{u}+1}}\cdots s_{n,i_{n}}$, $w=s_{1,i_{1}}\cdots s_{q_{_{n,r}}^{u}-1,i_{q_{_{n,r}}^{u}-1}}$.
Then by Lemma \ref{property Q}, we have  $m_{n}^{v}=r-1=m_{n}^{u}-1$ and $q_{n,j}^{v}=q_{n,j}^{u}  $ for any $1\leq j\leq r-1$.
Define
$A=\{(t_{m_{n}^{v}},\ldots, t_{1},  0 ) \mid  i_{n}-j\leq t_{j}\leq q_{n,j}^{v}+1, 1\leq j\leq m_{n}^{v}, t_{j}\in \mathbb{N}\}=\{(t_{m_{n}^{u}-1},\ldots, t_{1},0 ) \mid      i_{n}-j\leq t_{j}\leq q_{n,j}^{u}+1, 1\leq j\leq m_{n}^{u}-1, t_{j}\in \mathbb{N}\}$,
$B=\{(t_{m_{n}^{u}},\ldots, t_{1},0 ) \mid     i_{n}-j\leq t_{j}\leq q_{n,j}^{u}+1, 1\leq j\leq m_{n}^{u}, t_{j}\in \mathbb{N}\}$.

By induction hypothesis, we have
\begin{align*}
&\partial_{u}\mathfrak{S}_{w_{0}^{n+1}}&\\
=&\partial_{s_{1,i_{1}}\cdots s_{q_{_{n,r}}^{u}  ,i_{q_{n,r}^{u}  }}}\partial_{v}\mathfrak{S}_{w_{0}^{n+1}}&\\
=&\partial_{ws_{q_{_{n,r}}^{u}  ,i_{q_{n,r}^{u}  }}}
 (\sum_{(t_{m_{n}^{u}-1},\ldots, t_{1},0)\in A}
 (X[1,i_{n}-m_{n}^{v}-1]\prod_{1\leq j\leq m_{n}^{u}-1}x^{P(t_{j}>q_{n,j}^{u})}_{1+q_{n,j}^{u}}
\cdot \partial_{\left[\frac{s_{q_{n,r}^{u}  +1,i_{q_{n,r}^{u}  +1}}\cdots s_{n,i_{n}}}{(s_{t_{m_{n}^{u}  -1}},\ldots, s_{t_{2}},s_{t_{1}} )}\right]}
\mathfrak{S}_{w_{0}^{n}} ) )&\\
=&\sum_{(t_{m_{n}^{u}-1},\ldots, \cdots, t_{1},0)\in A} (
\prod_{1\leq j\leq m_{n}^{u}-1}x^{P(t_{j}>q_{n,j}^{u})}_{1+q_{_{n,j}}^{u}}
\partial_{ws_{q_{_{n,r}}^{u}  ,i_{q_{n,r}^{u}  }}}
(X[1,i_{n}-m_{n}^{u}]
\cdot \partial_{\left[\frac{s_{q_{n,r}^{u}  +1,i_{q_{n,r}^{u}  +1}}\cdots s_{n,i_{n}}}{(s_{t_{m_{n}^{u}  -1}},\ldots, s_{t_{2}},s_{t_{1}} )}\right]}
\mathfrak{S}_{w_{0}^{n}}) )&\\
=&\sum_{(t_{m_{n}^{u}-1},\ldots, \ldots, t_{1},0)\in A}
(\prod_{1\leq j\leq m_{n}^{u}-1}x^{P(t_{j}>q_{n,j}^{u})}_{1+q_{_{n,j}}^{u}}
\partial_{w}
 (\sum_{i_{n}-m_{n}^{u}  \leq t_{m_{n}^{u}  }\leq 1+q_{n,m_{n}^{u}  }^{u}  }(X[1,i_{n}-m_{n}^{u}-1]x^{P(t_{m_{n}^{u}  }>q_{n,m_{n}^{u}  }^{u}  )}_{1+q_{n,m_{n}^{u}  }^{u}}
&\\
&\cdot \frac{\partial_{q_{n,m_{n}^{u}   }^{u},i_{q_{n,m_{n}^{u}  }^{u} }}}{\partial_{t_{m_{n}^{u}  }}}
 \partial_{\left[\frac{s_{q_{n,r}^{u}  +1,i_{q_{n,r}^{u}  +1}}\cdots s_{n,i_{n}}}{(s_{t_{m_{n}^{u}  -1}},\ldots, s_{t_{2}},s_{t_{1}} )}\right]}
\mathfrak{S}_{w_{0}^{n}})))&\\
=&\sum_{(t_{m_{n}^{u}-1},\ldots, \cdots, t_{1},0)\in A}
(\prod_{1\leq j\leq m_{n}^{u}-1}x^{P(t_{j}>q_{n,j}^{u})}_{1+q_{_{n,j}}^{u}}
 (\sum_{i_{n}-m_{n}^{u}  \leq t_{m_{n}^{u}  }\leq 1+q_{n,m_{n}^{u}  }^{u}  } (X[1,i_{n}-m_{n}^{u}-1]x^{P(t_{m_{n}^{u}  }>q_{n,m_{n}^{u}  }^{u}  )}_{1+q_{n,m_{n}^{u}  }^{u}}
&\\
&\cdot \partial_{w}\frac{\partial_{q_{n,m_{n}^{u}  }^{u},i_{q_{n,m_{n}^{u}  }^{u}  }}}{\partial_{t_{m_{n}^{u}  }}}
 \partial_{\left[\frac{s_{q_{n,r}^{u}  +1,i_{q_{n,r}^{u}  +1}}\cdots s_{n,i_{n}}}{(s_{t_{m_{n}^{u}  -1}},\ldots, s_{t_{2}},s_{t_{1}} )}\right]}
\mathfrak{S}_{w_{0}^{n}})))  &\\
=&\sum_{\overrightarrow{T_{n}}=(t_{m_{n}^{u}},\ldots, t_{1},0)\in B}( X[1,i_{n}-m_{n}^{u}-1]
\prod_{1\leq j\leq m_{n}^{u}}x^{P(t_{j}>q_{n,j}^{u})}_{1+q_{_{n,j}}^{u}}
\cdot \partial_{u(\overrightarrow{T_{n}})}
\mathfrak{S}_{w_{0}^{n}})&\\
=&\sum_{\overrightarrow{T_{n}}}X_{\overrightarrow{T_{n}}}^{u}\partial_{u(\overrightarrow{T_{n}})}\mathfrak{S}_{w_{0}^{n}},&
\end{align*}
where     the  summation  is over all the
 $\overrightarrow{T_{n}}\in B$ such that $u(\overrightarrow{T_{n}})$ is defined.
By induction hypothesis, we have
\begin{align*}
&\partial_{u}\mathfrak{S}_{w_{0}^{n+1}}&\\
=&\sum_{\overrightarrow{T_{n}}}(X_{\overrightarrow{T_{n}}}^{u}\sum_{\overrightarrow{T_{n-1}}, \cdots, \overrightarrow{T_{1}}}X_{\overrightarrow{T_{n-1}}}^{u(\overrightarrow{T_{n}})}X_{\overrightarrow{T_{n-2}}}^{u(\overrightarrow{T_{n }})(\overrightarrow{T_{n-1}})}\cdots X_{\overrightarrow{T_{1}}}^{u(\overrightarrow{T_{n }})(\overrightarrow{T_{n-1}},\ldots, \overrightarrow{T_{2}} )})&\\
=&\sum_{\overrightarrow{T_{n }}, \cdots, \overrightarrow{T_{1}}}X_{\overrightarrow{T_{n }}}^{u}X_{\overrightarrow{T_{n-1}}}^{u(\overrightarrow{T_{n }})}\cdots X_{\overrightarrow{T_{1}}}^{u(\overrightarrow{T_{n }},\ldots, \overrightarrow{T_{2}})},&
\end{align*}
where the  summation  is over all the  $\overrightarrow{T_{n-1}}, \cdots, \overrightarrow{T_{1}}$ such that $u(\overrightarrow{T_{n-1}},\ldots, \overrightarrow{T_{l}})$ is defined for
any $1\leq l\leq n-1$.
\end{proof}

\begin{corollary}\emph{(\cite{Mac})}
For any $w\in S_{n}$, the coefficients of monomials in $ \mathfrak{S}_{w}$ are nonnegative integers.
\end{corollary}
\begin{proof}
 Let $u=[w^{-1}w_{0}^{n}]$. Then $ \mathfrak{S}_{w}=\partial_{u}\mathfrak{S}_{w_{0}^{n}}$.
 The result follows immediately from Lemma \ref{right normed delta} and Theorem \ref{formula P}. It also
 follows immediately from Theorem \ref{formula Q}.
\end{proof}
\begin{example} \label{example for formula Q}
$\partial_{u}\mathfrak{S}_{w_{0}^{5}}=   \sum\limits_{1\leq i\leq j\leq 3} x_{i}x_{j}$, where $u=s_{1,1}s_{2,1} s_{3,2} s_{4,2}$.
\end{example}
\begin{proof}
For $u=s_{1,1}s_{2,1} s_{3,2} s_{4,2}\in S_{5}$, by definition, we have  $q_{4,1}^{u}=2$, $m_{4}^{u}=1$,
$\overrightarrow{T_{4}}\in \{(1,0), (2,0), (3,0)\}$.

If $\overrightarrow{T_{4}}= (1,0)$,  then
$u(\overrightarrow{T_{4}})=s_{1,1}\cdot \frac{s_{2,1}}{s_{ 1}}\cdot s_{3,2}=s_{1,1}s_{2,2}s_{3,2}\in S_{4}$,
$X_{\overrightarrow{T_{4}}}^{u}=1$.

If $\overrightarrow{T_{4}}= (3,0)$,  then
$u(\overrightarrow{T_{4}})=s_{1,1}\cdot \frac{s_{2,1}}{s_{3}}\cdot s_{3,2}=s_{1,1}s_{2,1}s_{3,2}\in S_{4}$,
$X_{\overrightarrow{T_{4}}}^{u}=x_{3}$.

If $\overrightarrow{T_{4}}= (2,0)$,  then
$ s_{1,1}\cdot \frac{s_{2,1}}{s_{ 2}}\cdot s_{3,2}=s_{1}s_{1}s_{3,2} $ is not reduced,
  so $u(\overrightarrow{T_{4}})$ is undefined.

Let $\overrightarrow{T_{4}}= (1,0)$. Then $q_{3,1}^{u(\overrightarrow{T_{4}})}=1$,
$m_{3}^{u(\overrightarrow{T_{4}})}=1$,
 $\overrightarrow{T_{3}} \in \{(1,0),(2,0)\}$. If
 $\overrightarrow{T_{3}}=(1,0)$, then $u(\overrightarrow{T_{4}},\overrightarrow{T_{3}})=s_{2,2} \in S_{3}$,
 $X_{\overrightarrow{T_{3}}}^{u(\overrightarrow{T_{4}})}=1$. In this way, we can list all the possible
 $\overrightarrow{T_{4}},\ldots, \overrightarrow{T_{1}}$ such that
 $u(\overrightarrow{T_{4}},\ldots, \overrightarrow{T_{1}})$ is defined (Table \ref{table 2}).
 The result follows immediately from Theorem \ref{formula Q}.

\begin{table}[H]

\begin{tabular}
{|c|c|c|c|}

\hline

\multirow{12}{2.5cm}{$\overrightarrow{T_{4}}=(1,0)$,
$u(\overrightarrow{T_{4}})=s_{1,1}s_{2,2}s_{3,2}\in S_{4}$,
$X_{\overrightarrow{T_{4}}}^{u}=1$   }
& \multirow{4}{3.8cm}{ $\overrightarrow{T_{3}}=(1,0)$,
$u(\overrightarrow{T_{4}},\overrightarrow{T_{3}})= s_{2,2} \in S_{3}$,
$X_{\overrightarrow{T_{3}}}^{u(\overrightarrow{T_{4}})}=1$ }
 &  \multirow{4}{3.8cm}{ $\overrightarrow{T_{2}}=( 0)$,
$u(\overrightarrow{T_{4}},\ldots, \overrightarrow{T_{2}})= 1  \in S_{2}$,
$X_{\overrightarrow{T_{2}}}^{u(\overrightarrow{T_{4}},\overrightarrow{T_{3}})}=x_{1}$}
  &  \multirow{4}{3.8cm}{$\overrightarrow{T_{1}}=( 0)$,
$u(\overrightarrow{T_{4}},\ldots, \overrightarrow{T_{1}})= 1 \in S_{1}$,
$X_{\overrightarrow{T_{1}}}^{u(\overrightarrow{T_{4}},\ldots, \overrightarrow{T_{2}})}=x_{1}$ }      \\

 & & &  \\

 & & &  \\

 & & &  \\

   \cline{2-4}

& \multirow{8}{3.8cm}{ $\overrightarrow{T_{3}}=(2,0)$,
$u(\overrightarrow{T_{4}},\overrightarrow{T_{3}})= s_{1,1}s_{2,2} \in S_{3}$,
$X_{\overrightarrow{T_{3}}}^{u(\overrightarrow{T_{4}})}=x_{2}$}
 &  \multirow{4}{3.8cm}{ $\overrightarrow{T_{2}}=( 1,0)$,
$u(\overrightarrow{T_{4}},\ldots, \overrightarrow{T_{2}})= 1  \in S_{2}$,
$X_{\overrightarrow{T_{2}}}^{u(\overrightarrow{T_{4}},\overrightarrow{T_{3}})}=1$}
   &  \multirow{4}{3.8cm}{ $\overrightarrow{T_{1}}=( 0)$,
$u(\overrightarrow{T_{4}},\ldots, \overrightarrow{T_{1}})= 1 \in S_{1}$,
$X_{\overrightarrow{T_{1}}}^{u(\overrightarrow{T_{4}},\ldots, \overrightarrow{T_{2}})}=x_{1}$}      \\

 & & &  \\

 & & &  \\

 & & &  \\

  \cline{3-4}

&
 &  \multirow{4}{3.8cm}{ $\overrightarrow{T_{2}}=( 2,0)$,
$u(\overrightarrow{T_{4}},\ldots, \overrightarrow{T_{2}})= s_{1}  \in S_{2}$,
$X_{\overrightarrow{T_{2}}}^{u(\overrightarrow{T_{4}},\overrightarrow{T_{3}})}=x_{2}$ }
 &  \multirow{4}{3.8cm}{ $\overrightarrow{T_{1}}=( 0)$,
$u(\overrightarrow{T_{4}},\ldots, \overrightarrow{T_{1}})= 1 \in S_{1}$,
$X_{\overrightarrow{T_{1}}}^{u(\overrightarrow{T_{4}},\ldots, \overrightarrow{T_{2}})}=1$}      \\

 & & &  \\

 & & &  \\

 & & &  \\

  \cline{3-4}

\hline

\multirow{12}{2.5cm}{$\overrightarrow{T_{4}}=(3,0)$,
$u(\overrightarrow{T_{4}})=s_{1,1}s_{2,1}s_{3,2}\in S_{4}$,
$X_{\overrightarrow{T_{4}}}^{u}=x_{3}$   }
& \multirow{4}{3.8cm}{ $\overrightarrow{T_{3}}=(3,0)$,
$u(\overrightarrow{T_{4}},\overrightarrow{T_{3}})=s_{1,1} s_{2,1} \in S_{3}$,
$X_{\overrightarrow{T_{3}}}^{u(\overrightarrow{T_{4}})}=x_{3}$ }
 &  \multirow{4}{3.8cm}{ $\overrightarrow{T_{2}}=( 0)$,
$u(\overrightarrow{T_{4}},\ldots, \overrightarrow{T_{2}})= s_{1,1}$ $\in S_{2}$,
$X_{\overrightarrow{T_{2}}}^{u(\overrightarrow{T_{4}},\overrightarrow{T_{3}})}=1$}
  &  \multirow{4}{3.8cm}{$\overrightarrow{T_{1}}=( 0)$,
$u(\overrightarrow{T_{4}},\ldots, \overrightarrow{T_{1}})= 1 \in S_{1}$,
$X_{\overrightarrow{T_{1}}}^{u(\overrightarrow{T_{4}},\ldots, \overrightarrow{T_{2}})}=1$ }      \\

 & & &  \\

 & & &  \\

 & & &  \\

   \cline{2-4}

& \multirow{8}{3.8cm}{ $\overrightarrow{T_{3}}=(1,0)$,
$u(\overrightarrow{T_{4}},\overrightarrow{T_{3}})= s_{1,1}s_{2,2} \in S_{3}$,
$X_{\overrightarrow{T_{3}}}^{u(\overrightarrow{T_{4}})}=1$}
 &  \multirow{4}{3.8cm}{ $\overrightarrow{T_{2}}=( 1,0)$,
$u(\overrightarrow{T_{4}},\ldots, \overrightarrow{T_{2}})= 1  \in S_{2}$,
$X_{\overrightarrow{T_{2}}}^{u(\overrightarrow{T_{4}},\overrightarrow{T_{3}})}=1$}
   &  \multirow{4}{3.8cm}{ $\overrightarrow{T_{1}}=( 0)$,
$u(\overrightarrow{T_{4}},\ldots, \overrightarrow{T_{1}})= 1 \in S_{1}$,
$X_{\overrightarrow{T_{1}}}^{u(\overrightarrow{T_{4}},\ldots, \overrightarrow{T_{2}})}=x_{1}$}      \\

 & & &  \\

 & & &  \\

 & & &  \\

  \cline{3-4}

&
 &  \multirow{4}{3.8cm}{ $\overrightarrow{T_{2}}=( 2,0)$,
$u(\overrightarrow{T_{4}},\ldots, \overrightarrow{T_{2}})= s_{1}  \in S_{2}$,
$X_{\overrightarrow{T_{2}}}^{u(\overrightarrow{T_{4}},\overrightarrow{T_{3}})}=x_{2}$ }
 &  \multirow{4}{3.8cm}{ $\overrightarrow{T_{1}}=( 0)$,
$u(\overrightarrow{T_{4}},\ldots, \overrightarrow{T_{1}})= 1 \in S_{1}$,
$X_{\overrightarrow{T_{1}}}^{u(\overrightarrow{T_{4}},\ldots, \overrightarrow{T_{2}})}=1$}      \\

 & & &  \\

 & & &  \\

 & & &  \\

  \cline{3-4}
\hline

\end{tabular}
\caption{Example \ref{example for formula Q}}\label{table 2}
\end{table}

\end{proof}

\section{Some combinatorial properties of Schubert polynomials}\label{combinatorial property}
In this section, we will use Theorem \ref{formula Q} to develop some combinatorial properties of Schubert polynomials.

For any $u=s_{1,i_{1}}\cdots s_{n-1,i_{n-1}}\in S_{n}$,  $j\in [1,n-1]$, define
$$X_{j}^{u}=X[1,i_{j}-m_{j}^{u}-1]x_{1+q_{j,m_{j}^{u}}^{u}}x_{1+q_{j,m_{j}^{u}-1}^{u}}\cdots x_{1+q_{j,1}^{u}}.$$
In particular, if $Q_{j}^{u}=\emptyset$, then $m_{j}^{u}=0$ and
$X_{j}^{u}=X[1,i_{j}-1]$.
For any commutative word $W=x_{1}^{k_{1}}\cdots x_{n-1}^{k_{n-1}}$  (each $k_{i}\in \mathbb{N}$),
define $\operatorname{deg}_{x_{t}}(W)=k_{t}$. It is clear that for any $j,t,p\in [1,n-1],p<t$, we have $\operatorname{deg}_{x_{t}}(X_{j}^{u})\leq 1$ and $\operatorname{deg}_{x_{t}}(X_{p}^{u})=0$.

\begin{lemma}\label{common suffix}
Let $u=s_{1,i_{1}}\cdots s_{n-1,i_{n-1}}\in S_{n}$, $v=s_{1,j_{1}}\cdots s_{n-1,j_{n-1}}\in S_{n}$, $r\in [1,n-1]$.
If $i_{k}= j_{k}$ for any $k\geq r$, then
\begin{enumerate}
\item[(1)] $\operatorname{deg}_{x_{k}}(X_{t}^{u})=\operatorname{deg}_{x_{k}}(X_{t}^{v})$ for any  $t\in [r+1,n-1]$, $k\in [r+1, t]$.
\item[(2)] If  we have also $i_{r-1}\leq j_{r-1}$, then $\operatorname{deg}_{x_{r}}(X_{t}^{u})\geq \operatorname{deg}_{x_{r}}(X_{t}^{v})$ for any $t\geq r$.
 \end{enumerate}
 \end{lemma}
\begin{proof}
 To prove   $(1)$, we only need to show that $\operatorname{deg}_{x_{k}}(X_{t}^{u})\leq \operatorname{deg}_{x_{k}}(X_{t}^{v})$ for any  $t\in [r+1,n-1]$, $k\in [r+1, t]$.
 Note that $X_{t}^{u}=X[1,i_{t}-m_{t}^{u}-1]x_{1+q_{t,m_{t}^{u}}^{u}}x_{1+q_{t,m_{t}^{u}-1}^{u}}\cdots x_{1+q_{t,1}^{u}}$.
If $\operatorname{deg}_{x_{k}}(X_{t}^{u})=0$, we are done.  If $i_t=t+1$, then $\operatorname{deg}_{x_{k}}(X_{t}^{u})=\operatorname{deg}_{x_{k}}(X_{t}^{v})$.
So we may assume that $\operatorname{deg}_{x_{k}}(X_{t}^{u})=1$, $i_{t}\leq t$.

If $\operatorname{deg}_{x_{k}}(X[i_{t}-m_{t}^{u}-1])=1$, then $i_{t}-m_{t}^{u}-1\geq k\geq r+1$. By definition, we have
 $i_{t}-m_{t}^{u}\in [i_{q_{t,m_{t}^{u}}^{u}},q_{t,m_{t}^{u}}^{u}]$, so $q_{t,m_{t}^{u}}^{u}\geq i_{t}-m_{t}^{u}\geq r+2$, $Q_{t}^{u}\subseteq [r+2,n]\subseteq [r,n]$.
By Lemma \ref{property Q}, we have $q_{t,p}^{u}=q_{t,p}^{v}$ for any $1\leq p\leq m_{t}^{u}$. So $m_{t}^{v}\geq m_{t}^{u}$. Moreover, if $m_{t}^{v}> m_{t}^{u}$,
then $i_{t}-m_{t}^{u}-1\in [i_{q_{t,m_{t}^{u}+1}^{v}},q_{t,m_{t}^{u}+1}^{v}]$, $q_{t,m_{t}^{u}+1}^{v}\geq i_{t}-m_{t}^{u}-1\geq r+1$ but
$q_{t,m_{t}^{u}+1}^{v}\notin Q_{t}^{u}$,
which contradicts with Lemma
\ref{property Q}. Therefore $Q_{t}^{u}=Q_{t}^{v}$, $X_{t}^{u}=X_{t}^{v}$, $\operatorname{deg}_{x_{k}}(X_{t}^{u})\leq \operatorname{deg}_{x_{k}}(X_{t}^{v})$.

If $x_{k}=x_{q_{t,l}^{u}+1}$ for some $l\in [1,m_{t}^{u}]$, then
$q_{t,l}^{u}=k-1\geq r+1-1\geq r$. By Lemma \ref{property Q}, we
have $q_{t,l}^{u}\in Q_{t}^{u}\cap [r,n-1]=Q_{t}^{v}\cap [r,n-1]$ and
$q_{t,l}^{u}=q_{t,l}^{v}$. Hence $x_{q_{t,l}^{v}+1}=x_{q_{t,l}^{u}+1}=x_{k}$,
 $\operatorname{deg}_{x_{k}}(X_{t}^{u})\leq \operatorname{deg}_{x_{k}}(X_{t}^{v})$.

To prove $(2)$, we only need to show that if for some $t\geq r$,  $\operatorname{deg}_{x_{r}}(X_{t}^{v})=1$, then $ \operatorname{deg}_{x_{r}}(X_{t}^{u})=1$.
If $i_{t}=t+1$, then we are done. So we may assume that $i_{t}\leq t$.

If $\operatorname{deg}_{x_{r}}(X[j_{t}-m_{t}^{v}-1])=1$, then $j_{t}-m_{t}^{v}-1\geq r$. By definition, we have
 $j_{t}-m_{t}^{v}\in [j_{q_{t,m_{t}^{v}}^{v}},q_{t,m_{t}^{v}}^{v}]$, so $q_{t,m_{t}^{v}}^{v}\geq j_{t}-m_{t}^{v}\geq r+1$, $Q_{t}^{v}\subseteq [r+1,n]\subseteq [r,n]$.
By Lemma \ref{property Q}, we have $Q_{t}^{u}\cap [r,n]=Q_{t}^{v}\cap [r,n]=Q_{t}^{v}$. Moreover,
 $i_{t}-m_{t}^{v}-1=j_{t}-m_{t}^{v}-1\geq r$ and thus for any $q< r$, $i_{t}-m_{t}^{v}-1\notin [i_{q}, q]$.
So $Q_{t}^{u} =Q_{t}^{v}$, $X_{t}^{u}=X_{t}^{v}$, $ \operatorname{deg}_{x_{r}}(X_{t}^{u})=1$.

If $x_{r}=x_{q_{t,p}^{v}+1}$ for some $p\in [1,m_{t}^{v}]$, then $q_{t,p}^{v}=r-1$ and $q_{t,l}^{v}\geq r $ for any $l \leq p-1$.
Moreover, $i_{t}-p=j_{t}-p\in [j_{q_{t,p}^{v}}, q_{t,p}^{v}]=[j_{r-1},r-1]\subseteq [i_{r-1},r-1]$ and
$i_{t}-p=j_{t}-p\notin[j_{l},l]=[i_{l},l]$ for any $l\in [r , q_{t,p-1}^{v}-1]=[r , q_{t,p-1}^{u}-1]$. So $q_{t,p}^{u}=r-1$, $x_{r}=x_{q_{t,p}^{u}+1}$,
$ \operatorname{deg}_{x_{r}}(X_{t}^{u})=1$.
\end{proof}

Let $X=\{x_{1},\ldots, x_{n-1}\}$. Define an order $<$ on  the free commutative monoid $[X]$ as follows:
For any $U=x_{1}^{k_1}\cdots x_{n-1}^{k_{n-1}}\in [X]$, $V=x_{1}^{l_1}\cdots x_{n-1}^{l_{n-1}}\in [X]$,
$$
U<V \Leftrightarrow (\sum_{1\leq i\leq n-1}k_i, k_{n-1}, \ldots, k_{2}, k_1)>(\sum_{1\leq i\leq n-1}l_i, l_{n-1}, \ldots, l_{2}, l_1) \mbox{ lexicographically}.
$$
For any $f\in \mathbb{Z}[x_{1},\ldots, x_{n-1}]$, define $\bar{f}$ to be the leading monomial of $f$ with respect to the order $<$.  If the coefficient of $\bar{f}=1$, then we say that $f$ is monic. For example,
if $f=3x_{3}^{2}+2x_{3}x_{7}-7x_{5}x_{7} $, then $\bar{f}=x_{5}x_{7}$.

\begin{lemma}\label{write leading term}
For any $u=s_{1,i_{1}}\cdots s_{n-1,i_{n-1}}\in S_{n}$ $(n\geq 2)$, we
have $\partial_{u}\mathfrak{S}_{w_{0}^{n}}$ is monic and $\overline{\partial_{u}\mathfrak{S}_{w_{0}^{n}}}=X_{n-1}^{u}\cdots X_{1}^{u}$,
where $X_{j}^{u}=X[1,i_{j}-m_{j}^{u}-1]x_{1+q_{j,m_{j}^{u}}^{u}}x_{1+q_{j,m_{j}^{u}-1}^{u}}\cdots x_{1+q_{j,1}^{u}}$ for any $j\in [1,n-1]$.
\end{lemma}

\begin{proof}
Induction on $n$. If $n=2$, then $\partial_{u}\mathfrak{S}_{w_{0}^{n}}=X[1,i_{1}-1]=X_{1}^{u}$.
Suppose the lemma holds for any $k\leq n$. Let $k=n+1$,
 $u=s_{1,i_{1}}\cdots s_{n,i_{n}}\in S_{n+1}$, $u_{1}=s_{1,i_{1}}\cdots s_{n-1,i_{n-1}}\in S_{n}$.

If $m_{n}^{u}=0$, then  by the proof of Theorem \ref{formula Q}, we have
$$
\partial_{u}\mathfrak{S}_{w_{0}^{n+1}}
=\sum_{\overrightarrow{T_{n}}}X_{\overrightarrow{T_{n}}}^{u}\partial_{u(\overrightarrow{T_{n}})}(\mathfrak{S}_{w_{0}^{n}})
=X[1,i_{n}-1]\partial_{u_1}(\mathfrak{S}_{w_{0}^{n}})
=X_{n}^{u}\partial_{u_1}(\mathfrak{S}_{w_{0}^{n}}).
$$
By induction hypothesis, we have
$$
\overline{\partial_{u}\mathfrak{S}_{w_{0}^{n+1}}}
=X_{n}^{u}\overline{\partial_{u_1}\mathfrak{S}_{w_{0}^{n}}}
=X_{n}^{u}X_{n-1}^{u_{1}}\cdots X_{1}^{u_{1}}
=X_{n}^{u}X_{n-1}^{u }\cdots X_{1}^{u }.
$$

If $m_{n}^{u}>0$,
then  let $A=\{(t_{m_{n}^{u}},\ldots, t_{1},0 ) \mid     i_{n}-j\leq t_{j}\leq q_{n,j}^{u}+1, 1\leq j\leq m_{n}^{u}, t_{j}\in \mathbb{N}\}$.
 We have
$$
 \partial_{u}\mathfrak{S}_{w_{0}^{n+1}}
=
\sum_{\overrightarrow{T_{n}}=(t_{m_{n}^{u}},\ldots, t_{1},0 )\in A}
x^{P(t_{m_{n}^{u}  }>q_{n,m_{n}^{u}  }^{u}  )}_{1+q_{n,m_{n}^{u}   }^{u} }\cdots   x^{P(t_{1}>q_{n,1}^{u}  )}_{_{1+q_{n,1}^{u}  }}
X[1,i_{n}-m_{n}^{u}  -1]
 \partial_{u(\overrightarrow{T_{n}})}
\mathfrak{S}_{w_{0}^{n}}.
$$
So we   just need to show that if
$\overrightarrow{T_{n}}=(t_{m_{n}^{u}}, \cdots, t_{1},0)\neq (q_{n,m_{n}^{u}}^{u}+1,\ldots, q_{n,1}^{u}+1,0)$, then
$x^{P(t_{m_{n}^{u}  }>q_{n,m_{n}^{u}  }^{u}  )}_{_{_{1+q_{n,m_{n}^{u} }^{u} }  }}\cdots   x^{P(t_{1}>q_{n,1}^{u}  )}_{_{1+q_{n,1}^{u}  }}
X[1,i_{n}-m_{n}^{u}  -1]
 \overline{\partial_{u(\overrightarrow{T_{n}})}
(\mathfrak{S}_{w_{0}^{n}})}<
X_{n}^{u}\cdots X_{1}^{u}$ if $u(\overrightarrow{T_{n}})$ is defined.

Let $W(\overrightarrow{T_{n}})=X^{u}_{\overrightarrow{T_{n}}}X_{n-1}^{u(\overrightarrow{T_{n}})}\cdots X_{1}^{u(\overrightarrow{T_{n}})}
=
x^{P(t_{m_{n}^{u}  }>q_{n,m_{n}^{u}  }^{u}  )}_{_{_{1+q_{n,m_{n}^{u}  }^{u}}  }}\cdots   x^{P(t_{1}>q_{n,1}^{u}  )}_{_{1+q_{n,1}^{u}  }}
X[1,i_{n}-m_{n}^{u}  -1]\cdot X_{n-1}^{u(\overrightarrow{T_{n}})}\cdots X_{1}^{u(\overrightarrow{T_{n}})}$.
 Suppose that $t_{l}=q_{n,l}^{u}+1$ for any $1\leq l<r$ ($r\geq 1$) and $t_{r}\in [i_{q_{n,r}^{u}},q_{n,r}^{u}]$.
 Then  by the definition of $u(\overrightarrow{T_{n}})$, we have
 \begin{align*}
&u(\overrightarrow{T_{n}})&\\
 =&\left[\frac{s_{1,i_{1}}s_{2,i_{2}}\cdots s_{n-1,i_{n-1}}}{(s_{t_{m_{n }^{u}}},\ldots, s_{t_{2}},s_{t_{1}} )}\right]&\\
 =&\left[\frac{s_{1,i_{1}}s_{2,i_{2}}\cdots s_{q_{n,r}^{u}-1,i_{q_{n,r}^{u}-1}}}
 {(s_{t_{m_{n }^{u}}},\ldots, s_{t_{r+1}} )}
 \cdot \frac{s_{q_{n,r}^{u},i_{q_{n,r}^{u}}}}{s_{t_{r}}}\right]
 \cdot s_{q_{n,r}^{u}+1,i_{q_{n,r}^{u}+1}}\cdots s_{n-1,i_{n-1}}&\\
=&\left[\frac{s_{1,i_{1}}s_{2,i_{2}}\cdots s_{q_{n,r}^{u}-1,i_{q_{n,r}^{u}-1}}}
 {(s_{t_{m_{n }^{u}}},\ldots, s_{t_{r+1}} )}\cdot s_{t_{r}-1,i_{q_{n,r}^{u}}}\right]
 \cdot  s_{q_{n,r}^{u},t_{r}+1}
 \cdot s_{q_{n,r}^{u}+1,i_{q_{n,r}^{u}+1}}\cdots s_{n-1,i_{n-1}}&\\
=&s_{1,j_{1}}\cdots s_{n-1,j_{n-1}}\in S_{n}.&
\end{align*}
 By
  using the
Gr\"{o}bner-Shirshov basis of $S_{n}$, we have   $j_{t}=i_{t}$ for any $t\geq q_{n,r}^{u}+1$
and $i_{q_{n,r}^{u}}<t_{r}+1 =j_{q_{n,r}^{u}}$. By Lemma \ref{common suffix},
we have
 $\operatorname{deg}_{x_{1+q_{n,r}^{u}}}(X_{j}^{u})\geq \operatorname{deg}_{x_{1+q_{n,r}^{u}}}(X_{j}^{u(\overrightarrow{T_{n}})})$ for
 any $ n-1\geq j\geq q_{n,r}^{u}+1 $
 and
 $\operatorname{deg}_{x_{t}}(X_{j}^{u})=\operatorname{deg}_{x_{t}}(X_{j}^{u(\overrightarrow{T_{n}})})$  for any $n-1\geq j\geq t\geq q_{n,r}^{u}+2$.
Since  $\operatorname{deg}_{x_{q_{n,r}^{u}+1}}(X_{\overrightarrow{T_{n}}}^{u})=0<1=\operatorname{deg}_{x_{q_{n,r}^{u}+1}}(X_{n}^{u})$,
we have
 $\operatorname{deg}_{x_{t}}( X_{\overrightarrow{T_{n}}}^{u}
 \cdot X_{n-1}^{u(\overrightarrow{T_{n}})}\cdots X_{1}^{u(\overrightarrow{T_{n}})})=\operatorname{deg}_{x_{t}}(X_{n}^{u}\cdots X_{1}^{u})$
for any $t\geq q_{n,r}^{u}+2$ and
 $\operatorname{deg}_{x_{q_{n,r}^{u}+1}}( X_{\overrightarrow{T_{n}}}^{u}\cdot X_{n-1}^{u(\overrightarrow{T_{n}})}\cdots X_{1}^{u(\overrightarrow{T_{n}})})<\operatorname{deg}_{x_{q_{n,r}^{u}+1}}(X_{n}^{u}\cdots X_{1}^{u})$.
Since $\partial_{u}\mathfrak{S}_{w_{0}^{n}}$ is
homogeneous and the coefficients of $\partial_{u}\mathfrak{S}_{w_{0}^{n}}$ in Theorem \ref{formula Q} are nonnegative, the lemma
follows.
\end{proof}

\begin{lemma}\label{TFAE}
For any $u =s_{1,i_1}\cdots s_{n-1,i_{n-1}}\in S_{n}$, the following statements are equivalent:
\begin{enumerate}
\item[\emph{(i)}]  $s_{t}u $ is reduced.
\item[\emph{(ii)}] $i_{t-1}<i_{t}$.
\item[\emph{(iii)}]  $\operatorname{deg}_{x_{t}}(\overline{\partial_{u}\mathfrak{S}_{w_{0}^{n}}})>\operatorname{deg}_{x_{t+1}}(\overline{\partial_{u}\mathfrak{S}_{w_{0}^{n}}})$.
\end{enumerate}
Moreover, if $i_{t-1}<i_{t}$, then $\operatorname{deg}_{x_{t}}X_{t}^{u}=1$ and $\operatorname{deg}_{x_{t}}X_{j}^{u}\geq \operatorname{deg}_{x_{t+1}}X_{j}^{u}$ for any $j\in [t+1,n-1]$,
if  $i_{t-1}\geq i_{t}$, then $\operatorname{deg}_{x_{t}}X_{t}^{u}=0$ and $\operatorname{deg}_{x_{t}}X_{j}^{u}\leq  \operatorname{deg}_{x_{t+1}}X_{j}^{u}$ for any $j\in [t+1,n-1]$.
\end{lemma}
\begin{proof}
(i)$\Rightarrow $ (ii)  By the Gr\"{o}bner-Shirshov basis of $S_{n}$, we have
$s_{t}u=s_{1,i_{1}}\cdots s_{t-2,i_{t-2}}\cdot s_{t}\cdot s_{t-1,i_{t-1}}s_{t,i_{t}}\cdot s_{t+1,i_{t+1}}\cdots s_{n,i_{n}}$.
Suppose that  $i_{t-1}\geq i_{t}$.
If $i_{t-1}=t\geq i_{t}$, then
$$
s_{t}u=s_{1,i_{1}}\cdots s_{t-2,i_{t-2}}\cdot s_{t}\cdot s_{t,i_{t}}\cdot s_{t+1,i_{t+1}}\cdots s_{n,i_{n}}.
$$
So $s_{t}u$ is not reduced.
If $t-1\geq i_{t-1}\geq i_t$, then
\begin{align*}
&s_{t}u&\\
=&s_{1,i_{1}}\cdots s_{t-2,i_{t-2}}\cdot s_{t}\cdot s_{t-1,i_{t-1}}s_{t,i_{t}}\cdot s_{t+1,i_{t+1}}\cdots s_{n,i_{n}}&\\
=&s_{1,i_{1}}\cdots s_{t-2,i_{t-2}}\cdot s_{t-1}\cdot s_{t} s_{t-1,i_{t-1}}\cdot s_{t-1,i_{t}}\cdot s_{t+1,i_{t+1}}\cdots s_{n,i_{n}}&\\
=&\cdots&\\
=&s_{1,i_{1}}\cdots s_{t-2,i_{t-2}}\cdot s_{t-1, i_{t-1}+1}\cdot s_{t} s_{t-1,i_{t-1}}\cdot s_{i_{t-1}+1,i_{t}}\cdot s_{t+1,i_{t+1}}\cdots s_{n,i_{n}}&\\
=&s_{1,i_{1}}\cdots s_{t-2,i_{t-2}}\cdot s_{t-1, i_{t-1}}\cdot s_{t} s_{t-1,i_{t-1}}\cdot s_{i_{t-1},i_{t}}\cdot s_{t+1,i_{t+1}}\cdots s_{n,i_{n}}.&
\end{align*}
So $s_{t}u$ is not reduced.

Consequently, if $s_{t}u$ is   reduced,  then $i_{t-1}<i_{t}$.

(ii) $\Rightarrow$ (i)   If $i_{t-1}<i_{t}$, then  by similar reasoning as above, we have
$s_{t}u=s_{1,i_{1}}\cdots s_{t-2,i_{t-2}}\cdot s_{t-1, i_{t}-1}\cdot s_{t} s_{t-1,i_{t-1}}\cdot s_{t+1,i_{t+1}}\cdots s_{n,i_{n}}\in S_{n}$,
so $s_{t}u$ is reduced.

(ii) $\Rightarrow $ (iii)   Note that $X_{j}^{u}=X[1,i_{j}-m_{j}^{u}-1]x_{1+q_{j,m_{j}^{u}}^{u}}x_{1+q_{j,m_{j}^{u}-1}^{u}}\cdots x_{1+q_{j,1}^{u}}$.
Since $\overline{\partial_{u}\mathfrak{S}_{w_{0}^{n}}}=X_{n-1}^{u}\cdots X_{1}^{u}$, it is enough to
show that $\operatorname{deg}_{x_{t}}X_{j}^{u}\geq \operatorname{deg}_{x_{t+1}}X_{j}^{u}$ for any $j\in [t+1,n-1]$ and
$\operatorname{deg}_{x_{t}}X_{t}^{u}=1 \ (> 0=\operatorname{deg}_{x_{t+1}}X_{t}^{u})$.

If  $\operatorname{deg}_{x_{t+1}}(X[1,i_{j}-m_{j}^{u}-1])=1$ for some $j\geq t+1$, then  $\operatorname{deg}_{x_{t}}(X[1,i_{j}-m_{j}^{u}-1])=1$.
If $x_{t+1}=x_{1+q_{j,l}^{u}}$ for some $l\in [1,m_{j}^{u}]$, then $q_{j,l}^{u}=t$,
$i_{j}-l\in [i_{t},t]$,
so $i_{j}-l-1\in [i_{t}-1,t-1]\subseteq [i_{t-1},t-1]$.
By definition, $q_{j,l+1}^{u}=t-1$, and thus $x_{1+q_{j,l+1}^{u}}=x_{t}$.
By the above reasoning, we have $\operatorname{deg}_{x_{t}}X_{j}^{u}\geq \operatorname{deg}_{x_{t+1}}X_{j}^{u}$ for any $j\in [t+1,n-1]$.

If $i_{t}=t+1$, then $Q_{t}^{u}=\emptyset$ and $X_{t}^{u}=X[1,i_{t}-1]=X[1,t]$. So $\operatorname{deg}_{x_{t}}X_{t}^{u}=1$.
 If $i_{t}\leq t$, then $i_{t}-1\in [i_{t-1}, t-1]$. By definition, we have
 $q_{t,1}^{u}=t-1$, $x_{t}=x_{1+q_{t,1}^{u}}$.  So $\operatorname{deg}_{x_{t}}X_{t}^{u}=1$.

(iii) $\Rightarrow $ (ii)  We just need to show that if $i_{t-1}\geq i_{t}$, then
$\operatorname{deg}_{x_{t}}(\overline{\partial_{u}\mathfrak{S}_{w_{0}^{n}}})\leq \operatorname{deg}_{x_{t+1}}(\overline{\partial_{u}\mathfrak{S}_{w_{0}^{n}}})$.
Since $\overline{\partial_{u}\mathfrak{S}_{w_{0}^{n}}}=X_{n-1}^{u}\cdots X_{1}^{u}$, it is enough to
show that $\operatorname{deg}_{x_{t}}X_{j}^{u}\leq \operatorname{deg}_{x_{t+1}}X_{j}^{u}$ for any $j\in [t+1,n-1]$ and
$\operatorname{deg}_{x_{t}}X_{t}^{u}=0\ (=\operatorname{deg}_{x_{t+1}}X_{t}^{u})$.

If  $\operatorname{deg}_{x_{t}}(X[1,i_{j}-m_{j}^{u}-1])=1$ and $\operatorname{deg}_{x_{t+1}}(X[1,i_{j}-m_{j}^{u}-1])=0$ for some $j\geq t+1$.
Then $t=i_{j}-m_{j}^{u}-1$. Since $i_{t-1}\geq i_{t}$, we have $ i_{t}\leq i_{t-1}\leq t$.
If $m_{j}^{u}=0$, then for any $r\in [1, j-1]$, $t=i_{j}-m_{j}^{u}-1=i_{j}-1\notin [i_{r},r]$, which contradicts with  $t\in [i_{t},t]$. If
$m_{j}^{u}>0$, then $i_{j}-m_{j}^{u}=t+1\in [i_{q_{j,m_{j}^{u}}^{u}},q_{j,m_{j}^{u}}^{u}]$,
$q_{j,m_{j}^{u}}^{u}\geq t+1$. Moreover, by the definition of $m_{j}^{u}$, we know that
$t=i_{j}-m_{j}^{u}-1\notin [i_{r},r]$ for any $r\in [1, q_{j,m_{j}^{u}}^{u}-1]$, which contradicts with $t\in [i_{t},t]$. So
  if $\operatorname{deg}_{x_{t}}(X[1,i_{j}-m_{j}^{u}-1])=1$, then $\operatorname{deg}_{x_{t+1}}(X[1,i_{j}-m_{j}^{u}-1])=1$.

 If $x_{t}=x_{q_{j,l}^{u}+1}$ for some $l\in [1,m_{j}^{u}]$, then
 $i_{j}-l\in [i_{t-1},t-1]\subseteq [i_{t},t]$ and
 $i_{j}-l\notin [i_{q},q]$ for any $q\in [t,q_{j,l-1}^{u}-1]$. This is possible
 only if $q_{j,l-1}^{u}=t$, which means that $x_{t+1}=x_{q_{j,l-1}^{u}+1}$.
 By the above reasoning, we have $\operatorname{deg}_{x_{t}}X_{j}^{u}\leq \operatorname{deg}_{x_{t+1}}X_{j}^{u}$ for any $j\in [t+1,n-1]$.
Since $i_{t}\leq i_{t-1}\leq t$, we have
 $i_{t-1}> i_{t}-1$ and $i_{t}-1\leq t-1$. By the definition of $q_{t,1}^{u}$ and $X_{t}^{u}$, we have $t-1\notin Q_{t}^{u}$ and
 $\operatorname{deg}_{x_{t}}(X_{t}^{u})=0$.
\end{proof}



\begin{lemma}\label{leading term step by step}
For any $v=s_{1,i_1}\cdots s_{n-1,i_{n-1}}\in S_{n}$, if $\operatorname{deg}_{x_{t}}(\overline{\partial_{v}\mathfrak{S}_{w_{0}^{n}}})>\operatorname{deg}_{x_{t+1}}(\overline{\partial_{v}\mathfrak{S}_{w_{0}^{n}}})$,
then
$\overline{\partial_{s_{t}v}\mathfrak{S}_{w_{0}^{n}}}=\overline{\partial_{s_{t}}(\overline{\partial_{v}\mathfrak{S}_{w_{0}^{n}}})}$.
\end{lemma}
\begin{proof}
If  $\operatorname{deg}_{x_{t}}(\overline{\partial_{v}\mathfrak{S}_{w_{0}^{n}}})>\operatorname{deg}_{x_{t+1}}(\overline{\partial_{v}\mathfrak{S}_{w_{0}^{n}}})$,
then by Lemma \ref{TFAE}, we have $i_{t-1}<i_{t}$.
 Say $\overline{\partial_{v}\mathfrak{S}_{w_{0}^{n}}}=X_{1}^{v}\cdots X_{n-1}^{v}$. Let
 $$
 T=\{j\in [1,n-1]\mid  \operatorname{deg}_{x_{t}}X_{j}^{v}=1, \operatorname{deg}_{x_{t+1}}X_{j}^{v}=0 \}.
 $$
  Then by Lemma \ref{TFAE}, we have $t\in T$ and  $\operatorname{deg}_{x_{t}}X_{j}^{v}= \operatorname{deg}_{x_{t+1}}X_{j}^{v}$ for any $j\in [1,n-1]\setminus T$. We may assume
  that $T=\{p_{k}\mid 1\leq k\leq l\}\subseteq [t,n-1]$, $p_{1}=t$.   Then
 \begin{align*}
&\overline{\partial_{t}(\overline{\partial_{v}\mathfrak{S}_{w_{0}^{n}}})}&\\
 = &\overline{\partial_{t}(X_{1}^{v}\cdots X_{n-l}^{v})}&\\
 =&\overline{\partial_{t}(X_{q_{1}}^{v}\cdots X_{q_{n-l}}^{v}\cdot \frac{X_{p_{1}}^{v}}{x_{t}}        \cdots \frac{X_{p_{l}}^{v}}{x_{t}}x_{t}^{l})}&\\
 =&X_{q_{1}}^{v}\cdots X_{q_{n-l}}^{v}\cdot   \frac{X_{p_{1}}^{v}}{x_{t}}        \cdots \frac{X_{p_{l}}^{v}}{x_{t}} \overline{\partial_{t}(x_{t}^{l})}&\\
 =&X_{q_{1}}^{v}\cdots X_{q_{n-l}}^{v}\cdot   \frac{X_{p_{1}}^{v}}{x_{t}}        \cdots \frac{X_{p_{l}}^{v}}{x_{t}}  x_{t+1}^{l-1}&\\
 =&X_{q_{1}}^{v}\cdots X_{q_{n-l}}^{v}\cdot \frac{X_{p_{1}}^{v}}{x_{t}}\cdot (\frac{X_{p_{2}}^{v}}{x_{t}}\cdot x_{t+1})\cdots (\frac{X_{p_{l}}^{v}}{x_{t}}\cdot x_{t+1}),&
 \end{align*}
where $q_{1},\ldots, q_{n-l}\in [1,n-1]\setminus T$.

On the other hand,
let $u=s_{t}v$. By Lemma \ref{TFAE},   $u$ is a reduced word and
$u=s_{t}v=s_{1,i_{1}}\cdots s_{t-2,i_{t-2}}\cdot s_{t-1, i_{t}-1}\cdot s_{t} s_{t-1,i_{t-1}}\cdot s_{t+1,i_{t+1}}
\cdots s_{n,i_{n}}=s_{1,j_{1}}\cdots s_{n-1,j_{n-1}}\in S_{n}$, $j_{t}=i_{t-1}$,  $ j_{t-1}=i_{t}-1\leq t $, $j_{p}=i_{p}$
for any $p\in [1,n-1]\setminus \{t-1,t\}$.
Say $\overline{\partial_{u}\mathfrak{S}_{w_{0}^{n}}}=X_{1}^{u}\cdots X_{n-1}^{u}$.
By the definition of $X_{j}^{u}$ and Lemma \ref{property Q}, we have  $X_{j}^{u}=X_{j}^{v}$ for any $j\in [1,t-2]$.
The proof will proceed in steps.

(i) $X_{t}^{u}=X_{t-1}^{v}, X_{t-1}^{u}=\frac{X_{t}^{v}}{x_{t}}$.

  Since $u=s_{1,i_{1}}\cdots s_{t-2,i_{t-2}}\cdot s_{t-1, i_{t}-1}\cdot s_{t} s_{t-1,i_{t-1}}\cdot s_{t+1,i_{t+1}}
\cdots s_{n,i_{n}}, v=s_{1,i_1}\cdots s_{n-1,i_{n-1}}$ and $i_{t-1}<i_{t}$, it is straightforward to show  that
$Q_{t}^{v}=Q_{t-1}^{u}\cup \{t-1\}$, $Q_{t-1}^{v}=Q_{t}^{u}$, $q_{t,1}^{v}=t-1$.  Therefore
$X_{t}^{u}
=X[1,j_{t}-m_{t}^{u}-1]x_{1+q_{t, m_{t}^{u}}^{u}}\cdots x_{1+q_{t, 1}^{u}}
=X[1,i_{t-1}-m_{t-1}^{v}-1]x_{1+q_{t-1, m_{t-1}^{v}}^{v}}\cdots x_{1+q_{t-1, 1}^{v}}
=X_{t-1}^{v}$ and
$X_{t-1}^{u}
=X[1,j_{t-1}-m_{t-1}^{u}-1]x_{1+q_{t-1, m_{t-1}^{u}}^{u}}\cdots x_{1+q_{t-1, 1}^{u}}
=X[1,(i_{t}-1)-(m_{t}^{v}-1)-1]x_{1+q_{t, m_{t}^{v}}^{v}}\cdots x_{1+q_{t, 2}^{v}}
=\frac{X_{t}^{v}}{x_{t}}$.

(ii) For any $j\in [t+1,n-1]$,  if   $j\in T\setminus \{t\}$, then we have
    $X_{j}^{u}=\frac{X_{j}^{v}}{x_{t}}\cdot x_{t+1}$.

If for some    $j\in [t+1,n-1]$,      $j\in T\setminus \{t\}$,
$i_{j}=j+1$, then $X_{j}^{v}=X[1,j]$, which contradicts with $j\in T$. So we have $i_{j}\leq j$. By Lemma \ref{property Q}, we have $q_{j,m_{j}^{v}}^{v}>i_{j}-m_{j}^{v}-1$.
Note that
$X_{j}^{v}=X[1, i_{j}-m_{j}^{v}-1]x_{1+q_{j, m_{j}^{v}}^{v}}\cdots x_{1+q_{j, 1}^{v}}$. There are
two cases:

Case 1. If $x_{t}=x_{ i_{j}-m_{j}^{v}-1}$, then $ t=i_{j}-m_{j}^{v}-1\notin[i_{p},p]$ for any $ p\in [1, q_{j,m_{j}^{v}}^{v}-1]$ and $q_{j,m_{j}^{v}}^{v}-1\geq i_{j}-m_{j}^{v}-1=t$. In particular, $t\notin [i_{t},t]$, so $i_{t}=t+1$, $i_{t-1}\leq i_{t}-1=t$.
 We have
$i_{j}-(m_{j}^{v}+1)=t\in [i_{t-1},t]=[j_{t},t]$, $i_{j}-(m_{j}^{v}+1)\notin [i_{p},p]$ for any $p\in [t+1, q_{j,m_{j}^{v}}^{v}-1]$.
Therefore  $q_{j,m_{j}^{v}+1}^{u}= t$, $q_{j,k}^{u}=q_{j,k}^{v}$ for any $k\in [1,m_{j}^{v}]$. Moreover, $i_{j}-(m_{j}^{v} +1)-1=t-1\notin [t,t-1]=[i_{t}-1,t-1]$ and $t-1\notin [i_{p},p]$ for any $p\leq t-2$. It follows that $Q_{j}^{u}=Q_{j}^{v}\cup \{t\}$ and $X_{j}^{u}
 =X[1,i_{j}-(m_{j}^{v}+1)-1]\cdot x_{1+q_{j,m_{j}^{u}}^{u }}\cdot x_{1+q_{j,m_{j}^{v}}^{v}}\cdots x_{1+q_{j,1}^{v}}
=X[1,i_{j}-m_{j}^{v}-2]\cdot x_{t+1}\cdot x_{1+q_{j,m_{j}^{v }}^{v}}\cdots x_{1+q_{j,1}^{v}}=\frac{X_{j}^{v}}{x_{t}}\cdot x_{t+1}$.

Case 2. If $x_{t}=x_{1+q_{j,k}^{v}}$ for some $k\in [1,m_{j}^{v}]$, then
$q_{j,k}^{v}=t-1$. Since $j\in T$,  we have $q_{j,k-1}^{v}>t$,  $i_{j}-k\in [i_{t-1},t-1]$
and $i_{j}-k\notin \cup_{t\leq p\leq q_{j,k-1}^{v}-1}[i_{p},p]$.
Thus  $i_{j}-k\in [i_{t-1},t]=[j_{t},t]$
and $i_{j}-k\notin \cup_{t+1\leq p\leq q_{j,k-1}^{v}-1}[j_{p},p]$.
By definition, we have $q_{j,p}^{u}=q_{j,p}^{v}$ for any $p\in [1,k-1]$ and
$q_{j,k}^{u}=t$. Moreover, since $i_{j}-k\notin [i_{t},t]$, we have
$i_{j}-k-1\notin [i_{t}-1,t-1]=[j_{t-1},t-1]$. So $t-1\notin Q_{j}^{u}$.
It follows immediately that $q_{j,l}^{v}=q_{j,l}^{u}$ for any $l\in [k+1,m_{j}^{v}]$ and
$m_{j}^{v}=m_{j}^{u}$.
Therefore  $Q_{j}^{u}\setminus Q_{j}^{v}=\{t\}$, $Q_{j}^{v}\setminus Q_{j}^{u}=\{t-1\}$,
$X_{j}^{u}=\frac{X_{j}^{v}}{x_{t}}\cdot x_{t+1}$.

 (iii) For any $j\in [t+1,n-1]\setminus T$,      we have $X_{j}^{u}=X_{j}^{v}$.

 For any $j\in [t+1, n-1]\setminus T$, we have
 $\operatorname{deg}_{x_{t}}(X_{j}^{v})=\operatorname{deg}_{x_{t+1}}(X_{j}^{v})$.
 If for some    $j\in [t+1, n-1]\setminus T$,
$i_{j}=j+1$, then $X_{j}^{u}=X_{j}^{v}=X[1,j]$.
So we may assume that $i_{j}\leq j$. By Lemma \ref{property Q}, we have $q_{j,m_{j}^{v}}^{v}>i_{j}-m_{j}^{v}-1$.
 There are two cases:

 Case 1. If $\operatorname{deg}_{x_{t}}(X_{j}^{v})=\operatorname{deg}_{x_{t+1}}(X_{j}^{v})=0$, then
 $i_{j}-m_{j}^{v}-1\leq t-1$. Moreover, if $q_{j,m_{j}^{v}}^{v}\geq t+1$, then
 since $i_{j}-m_{j}^{v}-1\notin [i_{p},p]$ for any $p\in [1,q_{j,m_{j}^{v}}^{v}-1]$ by the definition of $m_{j}^{v}$, we have
    $i_{j}-m_{j}^{v}-1\notin [i_{t-1},t-1]\cup \{t\}$.
   Since $i_{t-1}\leq i_{t}-1$, we have
    $i_{j}-m_{j}^{v}-1\notin
 \cup_{1\leq p\leq t-2}[i_{p},p]\cup [i_{t}-1,t-1]\cup [i_{t-1},t] \cup  \cup_{t+1\leq p\leq q_{j,m_{j}^{v}}^{v}-1}[i_{p},p]$,
 so $m_{j}^{u}=m_{j}^{v}$, $Q_{j}^{u}=Q_{j}^{v}$, $X_{j}^{u}=X_{j}^{v}$.
 If $q_{j,m_{j}^{v}}^{v}\leq  t-2$, then there is some $k\in [0,m_{j}^{v}-1]$  such that
 $q_{j,k+1}^{v}\leq t-2$,  $q_{j,k}^{v}\geq  t+1$.
  Then $i_{j}-k-1\notin \cup_{q_{j,k+1}^{v}+1\leq p\leq q_{j,k}^{v}-1}[i_{p},p]$ and $i_{j}-k-1\in [i_{q_{j,k+1}^{v}},q_{j,k+1}^{v}]$.
   In particular,  $i_{j}-k-1\notin [i_{t-1},t-1]\cup \{t\}=[i_{t-1},t]\cup [i_{t}-1,t-1]$. So $\{t-1,t\}\cap Q_{j}^{u}=\emptyset$. By Lemma \ref{property Q}, we have
    $m_{j}^{u}=m_{j}^{v}$, $Q_{j}^{u}=Q_{j}^{v}$, $X_{j}^{u}=X_{j}^{v}$.

 Case 2.  $\operatorname{deg}_{x_{t}}(X_{j}^{v})=\operatorname{deg}_{x_{t+1}}(X_{j}^{v})=1$.
  If $i_{j}-m_{j}^{v}-1\geq t+1$, then $q_{j,m_{j}^{v}}^{v}>i_{j}-m_{j}^{v}-1\geq t+1$.
By similar reasoning as case 1, we have $m_{j}^{u}=m_{j}^{v}$, $Q_{j}^{u}=Q_{j}^{v}$, $X_{j}^{u}=X_{j}^{v}$.
If $i_{j}-m_{j}^{v}-1=t$, then  $q_{j,m_{j}^{v}}^{v}> i_{j}-m_{j}^{v}-1=t$, $q_{j,m_{j}^{v}}^{v}+1>t+1$, which
contradicts with $\operatorname{deg}_{x_{t+1}}(X_{j}^{v})=1$.
If $i_{j}-m_{j}^{v}-1<t$, then there is some $k\in [1,m_{j}^{v}-1]$ such that $1+q_{j,k}^{v}=t+1$, $1+q_{j,k+1}^{v}=t$. So
$i_{j}-k\in [i_{t},t]$, $i_{j}-k-1\in [i_{t-1},t-1]$. Consequently,
$i_{j}-k\in [i_{t-1},t]=[j_{t},t]$, $i_{j}-k-1\in [i_{t}-1,t-1]=[j_{t-1},t-1]$. By similar reasoning as case 1, we have  $m_{j}^{u}=m_{j}^{u}$, $Q_{j}^{u}=Q_{j}^{v}$, $X_{j}^{u}=X_{j}^{v}$.
\end{proof}


\begin{corollary}
For any reduced word $u=s_{t_{1}}\cdots s_{t_{p}}\in S_{n}$, we have
$$
\overline{\partial_{u}\mathfrak{S}_{w_{0}^{n}}}
=  \overline{\partial_{s_{t_{1}}} \overline{\partial_{s_{t_{2}}} \cdots(\overline{\partial_{s_{t_{p}}}\mathfrak{S}_{w_{0}^{n}}})}}.
$$
\end{corollary}

\begin{lemma}\label{distinct leading term}
For any $u,v\in S_{n}$, if $u\neq v$, then we have
$\overline{\partial_{u}\mathfrak{S}_{w_{0}^{n}}}\neq \overline{\partial_{v}\mathfrak{S}_{w_{0}^{n}}}$.
\end{lemma}
\begin{proof}
 If $|u|\neq |v|$, then since $\partial_{u}\mathfrak{S}_{w_{0}^{n}}$ is a
homogeneous polynomial of degree $\frac{1}{2}n(n-1)-|u|$, we have
$\overline{\partial_{u}\mathfrak{S}_{w_{0}^{n}}}\neq \overline{\partial_{v}\mathfrak{S}_{w_{0}^{n}}}$. Assume $|u|= |v|$.
Induction on $|u|$. If $|u|=1$, it is trivial.
Suppose
  $|u|= |v| \geq 2$, $u=s_{t}u_{1}\in S^{\ast}$, $v=s_{k}v_{1}\in S^{\ast}$,
  $u,v, u_{1}, v_{1}$ are all  normal forms in $S_{n}$ and
 $
 \overline{\partial_{u_1}\mathfrak{S}_{w_{0}^{n}}}=x_{1}^{l_1}\cdots x_{n-1}^{l_{n-1}}x_{n}^{l_{n}},\
 l_{t}>l_{t+1},\  l_{n}=0,
 $
 $
 \overline{\partial_{v_1}\mathfrak{S}_{w_{0}^{n}}}=x_{1}^{p_1}\cdots x_{n-1}^{p_{n-1}}x_{n}^{p_{n}},\
 p_{k}>p_{k+1},\  p_{n}=0.
 $
 By Lemma \ref{leading term step by step}, we have
 $$
 \overline{\partial_{u}\mathfrak{S}_{w_{0}^{n}}}
 =\overline{\partial_{s_{t}}\overline{\partial_{u_1}\mathfrak{S}_{w_{0}^{n}}}}
 =x_{1}^{l_1}\cdots x_{t-1}^{l_{t-1}} \cdot x_{t}^{l_{t+1}} x_{t+1}^{l_{t}-1} \cdot x_{t+2}^{l_{t+2}}\cdots x_{n-1}^{l_{n-1}}x_{n}^{l_{n}};
 $$
  $$
 \overline{\partial_{v}\mathfrak{S}_{w_{0}^{n}}}
 =\overline{\partial_{s_{k}}\overline{\partial_{v_1}\mathfrak{S}_{w_{0}^{n}}}}
 =x_{1}^{p_1}\cdots x_{k-1}^{p_{k-1}} \cdot x_{k}^{p_{k+1}} x_{k+1}^{p_{k}-1} \cdot x_{k+2}^{p_{k+2}} \cdots x_{n-1}^{p_{n-1}}x_{n}^{p_{n}}.
 $$

 If $t=k$, then $u_{1}\neq v_{1}$, by induction hypothesis, there exists some
 $q\in [1,n-1]$  such that $l_{q}\neq p_{q}$.  It follows immediately that
 $\overline{\partial_{u}\mathfrak{S}_{w_{0}^{n}}}\neq \overline{\partial_{v}\mathfrak{S}_{w_{0}^{n}}}$.
 If $t\neq k$, then we may assume that $t<k$.
 Since $u, v,u_{1}, v_{1} $ are all in normal form, we have $u=s_{t,i_{t}}\cdots s_{n-1,i_{n-1}}$, $u_{1}=s_{t-1,i_{t}}s_{t+1,i_{t+1}}\cdots s_{n-1,i_{n-1}} $,
 $v=s_{k,j_{k}}\cdots s_{n-1,j_{n-1}}$. Moreover,
  $s_{t}v=s_{t}s_{k,i_{k}}\cdots s_{n-1,i_{n-1}} $
  is a normal form, hence $s_{t}v$
  is reduced. By Lemmas \ref{TFAE} and \ref{leading term step by step}, we have
 $\operatorname{deg}_{x_{t}}(\overline{\partial_{u}\mathfrak{S}_{w_{0}^{n}}})=l_{t+1}\leq l_{t}-1
 =\operatorname{deg}_{x_{t+1}}(\overline{\partial_{u}\mathfrak{S}_{w_{0}^{n}}})$
 and
 $\operatorname{deg}_{x_{t}}(\overline{\partial_{v}\mathfrak{S}_{w_{0}^{n}}})>
 \operatorname{deg}_{x_{t+1}}(\overline{\partial_{v}\mathfrak{S}_{w_{0}^{n}}})$.
Consequently, $\overline{\partial_{u}\mathfrak{S}_{w_{0}^{n}}}\neq \overline{\partial_{v}\mathfrak{S}_{w_{0}^{n}}}$.
\end{proof}

Remind that $B_{x}:=\{x_{1}^{k_{1}}\cdots x_{n-1}^{k_{n-1}}\mid k_{i}+i\leq n, 1\leq i\leq n-1\}$.
Define the map $\varphi$  to be $\varphi: S_{n} \longrightarrow B_{x}$, $\varphi(u)=\overline{\partial_{u}\mathfrak{S}_{w_{0}^{n}}}$.
By Lemma \ref{distinct leading term}, we know that $\varphi$ is an injective map.  But the cardinal of  $B_{x}$ is $n!$, so
$\varphi$ is a  bijection. The inverse of $\varphi$ can be easily
 constructed by  Lemmas  \ref{divide degree bigger} and \ref{move one index}.

Combining   Lemmas  \ref{write leading term}-\ref{distinct leading term}, we have
\begin{theorem} \label{theorem property}
For any $u\in S_{n}$ $(n\geq 2)$,  $t\in [1,n-1]$, we have the following combinatorial properties
of Schubert polynomials:
\begin{enumerate}
\item[\emph{(i)}]  $\partial_{u}\mathfrak{S}_{w_{0}^{n}}$ is monic and $\overline{\partial_{u}\mathfrak{S}_{w_{0}^{n}}}=X_{n-1}^{u}\cdots X_{1}^{u}$,
where $X_{j}^{u}=X[1,i_{j}-m_{j}^{u}-1]\prod\limits_{1\leq k\leq m_{j}^{u}}x_{1+q_{_{j,k}}^{u}} $ for any $j\in [1,n-1]$.
Moreover, the map $\varphi: S_{n} \longrightarrow B_{x}$, $\varphi(u)=X_{n-1}^{u}\cdots X_{1}^{u}$ is a
bijection. In particular, for any $v\in S_{n}$, if $u\neq v$, then
$\overline{\partial_{u}\mathfrak{S}_{w_{0}^{n}}}\neq \overline{\partial_{v}\mathfrak{S}_{w_{0}^{n}}}$.
\item[\emph{(ii)}] $s_{t}u $ is reduced if and only if  $\operatorname{deg}_{x_{t}}(\overline{\partial_{u}\mathfrak{S}_{w_{0}^{n}}})>\operatorname{deg}_{x_{t+1}}(\overline{\partial_{u}\mathfrak{S}_{w_{0}^{n}}})$.
Moreover, if $\operatorname{deg}_{x_{t}}(\overline{\partial_{u}\mathfrak{S}_{w_{0}^{n}}})>\operatorname{deg}_{x_{t+1}}(\overline{\partial_{u}\mathfrak{S}_{w_{0}^{n}}})$,
    or equivalently, $s_{t}u $ is reduced,
then
$\overline{\partial_{s_{t}u}\mathfrak{S}_{w_{0}^{n}}}=\overline{\partial_{s_{t}}(\overline{\partial_{u}\mathfrak{S}_{w_{0}^{n}}})}$.
\end{enumerate}
\end{theorem}

For any $u\in S_{n}$, define
$$
\widetilde{u}:=\overline{\partial_{u}\mathfrak{S}_{w_{0}^{n}}}, \ \ \
 \mathfrak{S}_{\widetilde{u}}:=\partial_{u}\mathfrak{S}_{w_{0}^{n}}.
$$
Then we have

\begin{corollary}\label{commutative leading term}
For any $W \in B_{x}$, if $\operatorname{deg}_{x_{t}}(W)>\operatorname{deg}_{x_{t+1}}(W)$,  then $\partial_{t}\mathfrak{S}_{W}=\mathfrak{S}_{\overline{\partial_{t}W}}$.
\end{corollary}
\begin{proof}
Since $W\in B_{x}$, by Theorem \ref{theorem property}, there is a $u\in S_{n}$ such that $W=\overline{\partial_{u}\mathfrak{S}_{w_{0}^{n}}}$.
If   $\operatorname{deg}_{x_{t}}(W)>\operatorname{deg}_{x_{t+1}}(W)$,  then $s_{t}u$ is reduced and
$\overline{\partial_{t}\mathfrak{S}_{W}}
=\overline{\partial_{t}\partial_{u}\mathfrak{S}_{w_{0}^{n}}}
=\overline{\partial_{t}\overline{\partial_{u}\mathfrak{S}_{w_{0}^{n}}}}
=\overline{\partial_{t}W}
$.
By the definition of $\mathfrak{S}_{W}$, we are done.
\end{proof}

Since
$\partial_{u}\mathfrak{S}_{w_{0}^{n}}$ is monic, we easily get the following   corollaries.

\begin{corollary}\emph{(\cite{Manivel})}
The Schubert polynomials $\partial_{u}\mathfrak{S}_{w_{0}^{n}}$, as $u$ varies over all
permutations in $S_{n}$, form an additive basis of the  free $\mathbb{Z}$-module $\oplus_{b\in B_{x}}\mathbb{Z}b$.
\end{corollary}

\begin{corollary}\label{z basis} \emph{(\cite{Manivel})}
The Schubert polynomials $\mathfrak{S}_{\widetilde{u}}$,  as $u$ varies over all
permutations in $S_{\infty}$, form an additive basis of the free polynomial ring
$\mathbb{Z}[x_{1},x_{2}, \cdots ,x_{n},\cdots ]$.
\end{corollary}



\section{Algorithms for multiplication of Schubert polynomials}\label{section algorithms}

For any Schubert polynomials $\mathfrak{S}_{u}$, $\mathfrak{S}_{v}$,
by Corollary \ref{z basis}, we know that there are   structure constants $c_{u,v}^{w}\in \mathbb{Z}$ such that
 $$
 \mathfrak{S}_{u} \mathfrak{S}_{v}=\sum_{w} c_{u,v}^{w}\mathfrak{S}_{w}.
 $$
It is well known that the coefficients are all nonnegative (for example, see \cite{Fulton}),
but there is no combinatorial proof yet.

One of the most famous formula for multiplications of Schubert polynomials is   \textbf{Monk's formula} \cite{Monk,Fulton}:
$$
 \mathfrak{S}_{s_{k}} \mathfrak{S}_{w}=\sum_{v} \mathfrak{S}_{v},
$$
where the summation is over all $v$  such that $v=w\cdot s_{p}s_{p+1}\cdots  s_{q-2}s_{q-1,p}$ and
$l(v)=l(w)+1$, where  $p\leq k$ and  $q>k$. For example, $\mathfrak{S}_{s_{2}}\cdot\mathfrak{S}_{s_{2}} =\mathfrak{S}_{s_{1}s_{2}}+\mathfrak{S}_{s_{3}s_{2}}$.

We will offer   algorithms to calculate the  structure constants in the sequel. However, for simplicity of the algorithms, we will use the notation $\mathfrak{S}_{\widetilde{u}}$. Since $\mathfrak{S}_{s_{k}}=\sum\limits_{i\in [1,k]}x_{i}$, we have $\mathfrak{S}_{s_{k}}= \mathfrak{S}_{x_{k}}$.  Assume that $u\in S_{n-k}$. Then  by   Monk's formula, we have
$$
 \mathfrak{S}_{x_{k}} \mathfrak{S}_{\widetilde{u}}=\sum_{\widetilde{v}} c_{x_{k},\widetilde{u}}^{\widetilde{v}}\mathfrak{S}_{\widetilde{v}},
$$
where
 the summation is over all $\widetilde{v}$  such that $w_{0}^{n}v^{-1}=w_{0}^{n}u^{-1}\cdot s_{p}s_{p+1}\cdots  s_{q-2}s_{q-1,p}$,
$l(w_{0}^{n}v^{-1})=l(w_{0}^{n}u^{-1})+1$ where  $p\leq k$ and  $q>k$.

By Lemma \ref{leading term step by step}, we have the following lemma:
\begin{lemma}\label{divide degree 1}
Let $W=x_{1}^{j_{1}}\cdots x_{n-1}^{j_{n-1}}\in B_{x}$, $j_{n}=0$.
If for some $p_{1},\cdots , p_{m}\in [1,n-1]$, $p_{1}<p_{2}< \cdots < p_{m}$, $j_{p_{t}}=j_{p_{t}+1}+1$ for any $t\in [1,m]$, then
$\partial_{p_{m}}\cdots \partial_{p_{2}}\partial_{p_{1}}\mathfrak{S}_{W}=\mathfrak{S}_{V}$,
where $V=\frac{W}{x_{p_{1}}\cdots x_{p_{m}}}$.
\end{lemma}
\begin{proof}
Since $W\in B_{x}$, by Theorem \ref{theorem property}, there is a $u\in S_{n}$ such that $W=\overline{\partial_{u}\mathfrak{S}_{w_{0}^{n}}}$.
Induction on $m$. If $m=1$, then
$
\operatorname{deg}_{x_{p_{1}}}(\overline{\partial_{u}\mathfrak{S}_{w_{0}^{n}}})=j_{p_{1}}>
j_{p_{1}+1}=\operatorname{deg}_{x_{p_{1}+1}}(\overline{\partial_{u}\mathfrak{S}_{w_{0}^{n}}})
$.
So $\overline{\partial_{p_{1}}\mathfrak{S}_{W}}
=\overline{\partial_{p_{1}}\partial_{u}\mathfrak{S}_{w_{0}^{n}}}
=\overline{\partial_{p_{1}}\overline{\partial_{u}\mathfrak{S}_{w_{0}^{n}}}}
=\overline{\partial_{p_{1}}W}=V.
$
 By the definition of $\mathfrak{S}_{V}$, we have $\partial_{p_{1}}\mathfrak{S}_{W}=\mathfrak{S}_{V}$.

 Suppose the assertion holds for any  $k<m$. Then $\partial_{p_{m-1}}\cdots \partial_{p_{2}}\partial_{p_{1}}\mathfrak{S}_{W}=\mathfrak{S}_{V_1}$,  $\partial_{p_{m}}\mathfrak{S}_{V_1}=\mathfrak{S}_{V}$,
where $V_1=\frac{W}{x_{p_{1}}\cdots x_{p_{m-1}}}$ and
 $V=\frac{V_{1}}{x_{p_{m}}}=\frac{W}{x_{p_{1}}\cdots x_{p_{m}}}$.
\end{proof}

For any $W=x_{1}^{j_{1}}\cdots x_{n-1}^{j_{n-1}}\in B_{x}$,  $j_{1}\geq j_{2}\geq \cdots \geq j_{n-1}$, define
$$P_{k}^{W}=\{p\in [1,n-1]\mid  n-p-j_{p}\geq k\}=\{p_{k,1},\ldots, p_{k,t_{k}}\} \ \ (t_{k}=0 \mbox{ if  } P_{k}^{W}=\emptyset), $$
where $p_{k,i}<p_{k,j}$ if $i<j$. In particular, $P_{0}^{W}=[1,n-1]$.
If $P_{k}^{W}\neq \emptyset$ for some $k>0$, then define
$$
v_{k}^{W}=s_{p_{k,t_{k}}}\cdots s_{p_{k,1}}.
$$
Then we have the following lemma:

\begin{lemma}\label{divide degree bigger}
For any $W=x_{1}^{j_{1}}\cdots x_{n-1}^{j_{n-1}}\in B_{x}$,  $j_{1}\geq j_{2}\geq \cdots \geq j_{n-1}$, let
$m=\operatorname{max}\{k\in [0,n-1]\mid  P_{k}^{W}\neq \emptyset \}$. If $m>0$, then  for any $k\in [1,m]$, we have
$$
\partial_{v_{k}^{W}v_{k-1}^{W}\cdots v_{1}^{W}}\mathfrak{S}_{w_{0}^{n}}=
\frac{x_{1}^{n-1}x_{2}^{n-2}\cdots x_{n-1}^{1} }{x_{p_{k,t_{k}}}\cdots x_{p_{k,1}}\cdots x_{p_{2,t_{2}}}\cdots x_{p_{2,1}}x_{p_{1,t_{1}}}\cdots x_{p_{1,1}} }.
$$
In particular, $
\partial_{v_{m}^{W}v_{m-1}^{W}\cdots v_{1}^{W}}\mathfrak{S}_{w_{0}^{n}}=W$.
\end{lemma}
\begin{proof}
Let $V_{k}=\frac{x_{1}^{n-1}x_{2}^{n-2}\cdots x_{n-1}^{1} }{x_{p_{k,t_{k}}}\cdots x_{p_{k,1}}\cdots x_{p_{2,t_{2}}}\cdots x_{p_{2,1}}x_{p_{1,t_{1}}}\cdots x_{p_{1,1}} }$.
Induction on $k$.  If $k=1$, then by Lemma \ref{divide degree 1},
we have $\mathfrak{S}_{V_{1}}
=\partial_{p_{1,t_{1}}}\cdots \partial_{p_{1,1}}\mathfrak{S}_{x_{1}^{n-1}x_{2}^{n-2}\cdots x_{n-1}^{1}}
=\partial_{p_{1,t_{1}}}\cdots \partial_{p_{1,1}}(x_{1}^{n-1}x_{2}^{n-2}\cdots x_{n-1}^{1})=V_{1}$.

Suppose the lemma holds for any $q< k$, $k\geq 2$.
If $l, l+1\in P_{k}^{W}$, then  $l, l+1\in P_{k-1}^{W}$.
If  $l\in P_{k}^{W}$, $l+1\notin P_{k}^{W}$,  then  there is some integer $t$ such that
 $l+1 \in P_{t}^{W}$, $l+1\notin P_{t+1}^{W}$. It is clear that  $t\leq k-1$.
 Moreover, by the definition of $P_{k}^{W}$, we have $n-l-j_{l}\geq k$, $n-(l+1)-j_{l+1}=t$. So
 $n-(l+1)-t=j_{l+1}\leq j_{l}\leq n-l-k$. It follows that  $k-1\geq t\geq k-1$, i.e., $t=k-1$.
 Therefore  $l, l+1\in P_{k-1}^{W}$ and
$\operatorname{deg}_{x_{l}}(\partial_{ v_{k-1}^{W}\cdots v_{1}^{W}}\mathfrak{S}_{w_{0}^{n}})= n-l-(k-1)
=n-(l+1)-(k-1)+1=\operatorname{deg}_{x_{l+1}}(\partial_{ v_{k-1}^{W}\cdots v_{1}^{W}}\mathfrak{S}_{w_{0}^{n}})+1$.
By Lemma \ref{divide degree 1} and induction hypothesis, the lemma follows.
\end{proof}

\begin{lemma}\label{move one index}
  Given $t\in [0,n-2]$, $W_{t}=x_{1}^{j_{1}}\cdots x_{n-1}^{j_{n-1}}\in B_{x}$ such that
 $  j_{1}+1\geq j_{2}+2\geq \cdots \geq j_{t}+t\geq \operatorname{max}\{j_{k}+k\mid k\in [t+1,n-1]\}$, define
 $$
 v_{t+1}^{W_{t}}=s_{k-1,t+1},\ \   W_{t+1}=x_{1}^{j_{1}}\cdots x_{t}^{j_{t}}
 \cdot x_{t+1}^{j_{k}+k-t-1}
 \cdot \prod_{l\in [t+2,k]}x_{l}^{j_{l-1}} \cdot
 x_{k+1}^{j_{k+1}}\cdots x_{n-1}^{j_{n-1}},
 $$
where $k\in [t+1,n-1]$ is as small as possible such that $j_{k}+k=\operatorname{max}\{j_{k}+k\mid k\in [t+1,n-1]\}$. Then
  we have $\overline{\partial_{v_{_{t+1}}^{W_{t}}}\mathfrak{S}_{W_{t+1}}}=W_{t}$.

 In particular, for any $W_{0}\in B_{x}$, we can construct
 $ v_{1}^{W_{0}},W_{1},  \cdots,v_{n-2}^{W_{n-3}}, W_{n-2}$ such that $\overline{\partial_{v_{i}^{W_{i-1}}}\mathfrak{S}_{W_{i}}}=W_{i-1}$, which implies that
 $\overline{\partial_{v_{_{1}}^{W_{0}}}\cdots
 \partial_{v_{_{n-3}}^{W_{n-4}}}\partial_{v_{_{n-2}}^{W_{n-3}}}\mathfrak{S}_{W_{n-2}}}=W_{0}$. Moreover, if
 $W_{i}=x_{1}^{l_{i,1}}\cdots x_{n-1}^{l_{i, n-1}}$, then $l_{i,1}+1\geq l_{i,2}+2\geq \cdots \geq l_{i,i}+i\geq \operatorname{max}\{l_{i,k}+k\mid k\in [i+1,n-1]\}$.
\end{lemma}

\begin{proof}
Induction on $k-t$.
If $k=t+1$, then $W_{t+1}=W_{t}$, $v_{t+1}^{W_{t}}=s_{k-1,t+1}=1$,   $\overline{\partial_{v_{_{t+1}}^{W_{t}}}\mathfrak{S}_{W_{t+1}}}=W_{t}$.
If $k>t+1$. Let $W^{\prime}=x_{1}^{j_{1}}\cdots x_{k-2}^{j_{k-2}}
\cdot x_{k-1}^{j_{k}+1}x_{k }^{j_{k -1}} \cdot x_{k+1 }^{j_{k+1 }}\cdots
 x_{n-1}^{j_{n-1}}\in B_{x}$. Since $j_{k}+k>j_{k-1}+k-1$, we have $j_{k}+1>j_{k-1}$. By Corollary \ref{commutative leading term}, we have $\overline{\partial_{k-1}\mathfrak{S}_{W'}}=W_{t }$. By induction hypothesis, we have $\overline{\partial_{ s_{k-2,t+1}}\mathfrak{S}_{W_{t+1}}}=W^{\prime}$, i.e., $ \partial_{ s_{k-2,t+1}}\mathfrak{S}_{W_{t+1}}=\mathfrak{S}_{W^{\prime}}$. Consequently,
 $ \overline{\partial_{ s_{k-1,t+1}}\mathfrak{S}_{W_{t+1}}}
 =\overline{\partial_{s_{k-1}}\mathfrak{S}_{W^{\prime}}}
 =W_{t}$. The lemma follows immediately.
\end{proof}

\begin{example}
 $\mathfrak{S}_{x_{k}}=x_{1}+x_{2}+\cdots+x_{k}$.
\end{example}
\begin{proof}
 $\mathfrak{S}_{x_{k}}
=\partial_{k-1} (\mathfrak{S}_{x_{k-1}^{2}})
   =\partial_{k-1}\partial_{k-2} (\mathfrak{S}_{x_{k-2}^{3}})
 =\cdots
 =\partial_{k-1,1}(\mathfrak{S}_{x_{1}^{k}})
 =\partial_{k-1,1}(x_{1}^{k})
=x_{1}+x_{2}+\cdots+x_{k} $.
\end{proof}

\begin{example}
$\mathfrak{S}_{x_{3}x_{4}}=\partial_{s_{1,1}s_{2,1} s_{3,1} s_{4,3}}\mathfrak{S}_{w_{0}^{5}}.$
\end{example}
 \begin{proof}
 $\mathfrak{S}_{x_{3}x_{4}}
 =\partial_{3}\mathfrak{S}_{x_{3}^{2}x_{4}}
 =\partial_{3,2} \mathfrak{S}_{x_{2}^{3}x_{4}}
  =\partial_{3,1} \mathfrak{S}_{x_{1}^{4}x_{4}}
  =\partial_{3,1} \partial_{3}\mathfrak{S}_{x_{1}^{4}x_{3}^{2}}
  =\partial_{3,1} \partial_{3,2} \mathfrak{S}_{x_{1}^{4}x_{2}^{3}}
   =\partial_{3,1} \partial_{3,2} \partial_{3}\partial_{4,3}\mathfrak{S}_{w_{0}^{5}}
   =\partial_{u}\mathfrak{S}_{w_{0}^{5}},
 $
 where $u=s_{3,1}s_{3,2}s_{3}s_{4,3}$. By the Gr\"{o}bner-Shirshov basis of $S_{6}$, we have $[u]=s_{1,1}s_{2,1} s_{3,1} s_{4,3}$.
\end{proof}

In fact, Lemmas \ref{divide degree bigger}, \ref{move one index}  together offer an algorithm to
construct a reduced word $u\in S_{n}$  such that $\overline{\partial_{u}\mathfrak{S}_{w_{0}^{n}}}$  equals
an arbitrary commutative word in $B_{x}$.  And the
Gr\"{o}bner-Shirshov basis of $S_{n}$ offers an algorithm to rewrite any word $u\in S_{n}$ to its
normal form, which can be easily applied to calculate the Schubert Polynomial by Theorem \ref{formula P} or Theorem \ref{formula Q}.
Moreover, Lemma \ref{write leading term} offers an algorithm to write down the leading monomial
of the Schubert polynomial $\partial_{u}\mathfrak{S}_{w_{0}^{n}}$ for any $u\in S_{n}$.
For any $W\in B_{x}$, $f\in \mathbb{Z}[x_{1},\cdots x_{n}]$, denote by $c_{_{W}}(f)$  the coefficient of $W$ in $f$.
Recall that  for any $u\in S_{n}$,
$$
\widetilde{u}=\overline{\partial_{u}\mathfrak{S}_{w_{0}^{n}}}, \
 \mathfrak{S}_{\widetilde{u}}=\partial_{u}\mathfrak{S}_{w_{0}^{n}}.
$$

Now we can offer an algorithm to calculate the structure constants $c_{\widetilde{u},\widetilde{v}}^{\widetilde{w}}$ as follows:

\textbf{Algorithm 1}. For any $u,v\in S_{\infty}$,  by Lemma \ref{write leading term}, we have
$\overline{ \mathfrak{S}_{\widetilde{u}} \mathfrak{S}_{\widetilde{v}}}= \widetilde{u}\widetilde{v} $ and $\mathfrak{S}_{\widetilde{u}} \mathfrak{S}_{\widetilde{v}}$ is monic. By Lemma \ref{move one index} and   then Lemma \ref{divide degree bigger}, we can find a $w_1\in S_{\infty}$, such that $\widetilde{w_1}=\widetilde{u}\widetilde{v}$.  Say $w_1\in S_{n}$. Then using the Gr\"{o}bner-Shirshov basis of $S_{n}$, we can write down the normal form of $w_{1}$.
By Theorem \ref{formula Q} (or Theorem \ref{formula P}), we calculate
$\mathfrak{S}_{\widetilde{w_{1}}}=\partial_{w_{1}}\mathfrak{S}_{w_{0}^{n}}$.
Then we construct an $w_{2}$ such that  $\widetilde{w_{2}}=\overline{ \mathfrak{S}_{\widetilde{u}} \mathfrak{S}_{\widetilde{v}}-\mathfrak{S}_{\widetilde{w_{1}}}}<\widetilde{u}\widetilde{v}$.
 Since each $\mathfrak{S}_{\widetilde{w}}$ for any $w\in S_{\infty}$ is monic and   the leading monomials decrease in each step,  the algorithm works.

 In particular,
For any $u,v\in S_{\infty}$,  we have
$$
 \mathfrak{S}_{\widetilde{u}} \mathfrak{S}_{\widetilde{v}}= \mathfrak{S}_{\widetilde{u}\widetilde{v}}+ \sum_{|\widetilde{w}|=|\widetilde{u}\widetilde{v}|,\widetilde{w}< \widetilde{u}\widetilde{v}} c_{\widetilde{u},\widetilde{v}}^{\widetilde{w}}\mathfrak{S}_{\widetilde{w}}.
 $$
where
  $c_{\widetilde{u},\widetilde{v}}^{\widetilde{w}}
  =c_{\widetilde{w}}(\mathfrak{S}_{\widetilde{u}}\mathfrak{S}_{\widetilde{v}})
  -\sum\limits_{\widetilde{z} \in\{\widetilde{z}\mid \widetilde{w}<\widetilde{z}\leq\widetilde{u}\widetilde{v} \}}
  c_{\widetilde{u},\widetilde{v}}^{\widetilde{z}}c_{\widetilde{w}}(\mathfrak{S}_{\widetilde{z}})
  $
  if
  $ \widetilde{w}<\widetilde{u}\widetilde{v} $.

 \ \

By applying   Monk's formula and Lemma \ref{write leading term}, we  also have another algorithm to calculate the structure constants  $c_{\widetilde{u},\widetilde{v}}^{\widetilde{w}}$  as follows:

\textbf{Algorithm 2}. Induction on min$\{|\widetilde{u}|, |\widetilde{v}|\}$. If min$\{|\widetilde{u}|, |\widetilde{v}|\}\leq 1$, then by Monk's formula, we are done.
If $min\{|\widetilde{u}|, |\widetilde{v}|\}=t\geq 2$,
 say $|\widetilde{u}|=t$.
 Induction on $\widetilde{u}$. If $\widetilde{u}=x_{1}^{t}$, then $\mathfrak{S}_{\widetilde{u}}=x_{1}^{t}=\mathfrak{S}_{x_{1}}\cdot \mathfrak{S}_{x_{1}^{t-1}}$. By induction hypothesis, we have formula
 $$
 \mathfrak{S}_{\widetilde{u}}\mathfrak{S}_{\widetilde{v}}
 =\mathfrak{S}_{x_{1}}\cdot \mathfrak{S}_{x_{1}^{t-1}}\mathfrak{S}_{\widetilde{v}}
 =\mathfrak{S}_{x_{1}}\cdot (\sum_{\widetilde{z}} c_{ x_{1}^{t-1},\widetilde{v}}^{\widetilde{z}}\mathfrak{S}_{\widetilde{z}})
=\sum_{\widetilde{w},{\widetilde{z}} }  c_{x_{1},\widetilde{z}}^{\widetilde{w}}
 c_{ x_{1}^{t-1},\widetilde{v}}^{\widetilde{z}}\mathfrak{S}_{\widetilde{w}}.
 $$
If $\widetilde{u}>x_{1}^{t}$, then $\operatorname{deg}_{x_{k}}\widetilde{u}\geq 1$ for some $k$ (for example,
 choose the smallest one).
By Monk's formula, we have
\begin{align*}
&\mathfrak{S}_{\widetilde{u}}\mathfrak{S}_{\widetilde{v}}&\\
= &(\mathfrak{S}_{x_{k}} \mathfrak{S}_{\frac{\widetilde{u}}{x_{k}}}-\sum_{\widetilde{w}\in \{ \widetilde{w}\mid\widetilde{w}<\widetilde{u}\}} c_{x_{k},\frac{\widetilde{u}}{x_{k}}}^{\widetilde{w}}\mathfrak{S}_{\widetilde{w}}) \mathfrak{S}_{\widetilde{v}}&\\
=&\mathfrak{S}_{x_{k}} \mathfrak{S}_{\frac{\widetilde{u}}{x_{k}}}\mathfrak{S}_{\widetilde{v}}-\sum_{\widetilde{w}\in \{ \widetilde{w}\mid\widetilde{w}<\widetilde{u}\}} c_{x_{k},\frac{\widetilde{u}}{x_{k}}}^{\widetilde{w}}\mathfrak{S}_{\widetilde{w}} \mathfrak{S}_{\widetilde{v}}&\\
=&\mathfrak{S}_{x_{k}} (\sum_{\widetilde{w}} c_{\frac{\widetilde{u}}{x_{k}}, \widetilde{v}}^{\widetilde{w}}\mathfrak{S}_{\widetilde{w}})
-\sum_{\widetilde{w}\in \{ \widetilde{w}\mid\widetilde{w}<\widetilde{u}\}} c_{x_{k},\frac{\widetilde{u}}{x_{k}}}^{\widetilde{w}}\mathfrak{S}_{\widetilde{w}}  \mathfrak{S}_{\widetilde{v}}&\\
=&\sum_{\widetilde{w},\widetilde{z}} c_{x_{k}, \widetilde{w}}^{\widetilde{z}}  c_{\frac{\widetilde{u}}{x_{k}}, \widetilde{v}}^{\widetilde{w}}\mathfrak{S}_{\widetilde{z}}
-\sum_{\widetilde{w}\in \{ \widetilde{w}\mid\widetilde{w}<\widetilde{u}\}} \sum_{\widetilde{z}} c_{x_{k},\frac{\widetilde{u}}{x_{k}}}^{\widetilde{w}} c_{\widetilde{w},\widetilde{v}}^{\widetilde{z}}\mathfrak{S}_{\widetilde{z}}.&
\end{align*}



\end{document}